\documentclass[12pt]{amsart}
\usepackage{amsmath}
\usepackage{amsxtra}
\usepackage{amstext}
\usepackage{amssymb}
\usepackage{amsthm}
\usepackage{latexsym}
\usepackage{dsfont} 
\usepackage{verbatim}
\usepackage{tabls}
\usepackage{rotating}

\theoremstyle{plain}
\newtheorem{thm}{Theorem}[section]
\newtheorem{cor}[thm]{Corollary}
\newtheorem{lem}[thm]{Lemma}
\newtheorem{prop}[thm]{Proposition}

\newcommand{\norm}[1]{\left\Vert#1\right\Vert}
\theoremstyle{definition}
\newtheorem{defn}[thm]{Definition}
\newtheorem{exa}[thm]{Example}
\newtheorem{rem}[thm]{Remark}

\numberwithin{equation}{section}
\def\Lo{\smash{L_0^{1, \, 2}}}

\def\Xint#1{\mathchoice
   {\XXint\displaystyle\textstyle{#1}}%
   {\XXint\textstyle\scriptstyle{#1}}%
   {\XXint\scriptstyle\scriptscriptstyle{#1}}%
   {\XXint\scriptscriptstyle\scriptscriptstyle{#1}}%
   \!\int}
\def\XXint#1#2#3{{\setbox0=\hbox{$#1{#2#3}{\int}$}
     \vcenter{\hbox{$#2#3$}}\kern-.5\wd0}}

\def\dashint{\Xint-}

\begin{document}

\title[Existence and regularity of positive solutions]
{Existence and regularity of positive solutions to elliptic equations of Schr\"{o}dinger type}

\author[B.~J. Jaye]
{B.~J. Jaye}
\address{Department of Mathematics,
University of Missouri,
Columbia, MO 65211, USA}
\curraddr{Department of Mathematics,
Kent State University,
Kent, OH 44240, USA}
\email{bjaye@kent.edu}

\author[V.~G. Maz'ya]
{V.~G. Maz'ya}

\address{Department of Mathematical Sciences,
University of Liverpool,
Liverpool, L69 3BX, UK, and Department of Mathematics,
Link\"oping University, SE-581 83, Link\"oping,
Sweden}
\email{vlmaz@mai.liu.se}

\author[I.~E. Verbitsky]
{I.~E. Verbitsky}
\address{Department of Mathematics,
University of Missouri,
Columbia, MO 65211, USA}
\email{verbitskyi@missouri.edu}

\date{\today}

\thanks{The first and third authors are supported in part
by NSF grant  DMS-0901550.}


\begin{abstract}  We prove the existence of positive solutions with optimal local regularity to the homogeneous equation of Schr\"{o}dinger type,
$$-\text{div} (\mathcal{A}\nabla u) -\sigma u = 0 \quad \text{ in }\Omega,
$$
under only a form boundedness assumption on $\sigma \in D'(\Omega)$ and
ellipticity assumption on $\mathcal{A}\in L^\infty(\Omega)^{n\times n}$, for an arbitrary open set $\Omega\subseteq \mathbf{R}^n$.

We demonstrate  that there is a two way correspondence between the form boundedness  and the existence of positive  solutions to this equation, as well as weak
solutions to the equation with quadratic nonlinearity in the gradient,
$$-\text{div}(\mathcal{A}\nabla v )= (\mathcal{A}\nabla v)\cdot\nabla v + \sigma  \quad \text{ in }\Omega.
$$
 As a consequence, we obtain necessary
and sufficient conditions for both the form-boundedness (with a sharp upper
form bound)
and  the positivity
of the quadratic form of the Schr\"{o}dinger type operator
$\mathcal{H}=-\text{div} (\mathcal{A}\nabla \cdot ) - \sigma$ with arbitrary distributional
potential  $\sigma \in D'(\Omega)$,
and give examples clarifying the relationship between these two properties.  \end{abstract}

\maketitle

\section{Introduction}

\subsection{} The goal of this paper is to present an existence and regularity theory for \textit{positive} solutions to the equation of Schr\"{o}dinger type:
\begin{equation}\label{schrointro}
-\text{div}(\mathcal{A}\nabla u) - \sigma u=0\text{ in }\Omega,
\end{equation}
on an arbitrary open set $\Omega\subseteq \mathbf{R}^n$, $n \ge 1$,
under the standard ellipticity assumptions on $\mathcal{A}\in L^\infty(\Omega)^{n\times n}$, and the sole condition of form boundedness on the real-valued distributional potential $\sigma \in  D'(\Omega)$:
\begin{equation}\label{imbintro}
|\langle \sigma, h^2 \rangle| \leq C \int_{\Omega} (\mathcal{A}\nabla h )\cdot \nabla h \, dx,  \text{ for all  }h\in C^{\infty}_0(\Omega).
\end{equation}

Simultaneously, a corresponding theory  will be developed for  (possibly sign changing) weak solutions to the equation with quadratic growth in the gradient:
\begin{equation}\label{ricintro}
-\text{div}(\mathcal{A}\nabla v )= (\mathcal{A}\nabla v)\cdot\nabla v + \sigma \text{ in } \Omega.
\end{equation}
In displays (\ref{schrointro})--(\ref{ricintro}), $\mathcal{A}: \Omega \rightarrow \mathbf{R}^{n\times n}$ is a real $n\times n$ (possibly non-symmetric) matrix-valued function on $\Omega$, so that there exist $m, M>0$ such that, for almost every $x\in\Omega$
\begin{equation}\label{elliptic}
m|\xi|^2\leq \mathcal{A}(x)\xi\cdot \xi, \text{ and } |\mathcal{A}(x )\xi|\leq M|\xi|,  \text{ for all }\xi\in \mathbf{R}^n.
\end{equation}

It has been a long standing problem to extend the existing theory to general classes of  $\sigma$,  including highly oscillating, singular or distributional
potentials, where  the separation of the positive and negative parts of $\sigma$ is impossible due to
 the interaction between them. In our framework of distributional  $\sigma$, positive solutions to the Schr\"odinger equation are not locally bounded, and consequently standard PDE tools based on Harnack's inequality, and
the classical iterative techniques of Moser \cite{Mos60} and
Trudinger \cite{Tru73}, or  their
extension due to Brezis and Kato \cite{BK79}, are no longer available.

A primary result of the present paper is the following principle: \textit{For any form bounded potential $\sigma$ with the upper form bound strictly less than $1$, one can find a positive solution of (\ref{schrointro}) which lies in the local Sobolev space} $L^{1,2}_{\text{loc}}(\Omega)$.
This regularity for positive solutions of (\ref{schrointro}) is in fact optimal in the generality of the potentials considered here, as demonstrated by  examples discussed below.   Furthermore, there is a two way correspondence between the existence of positive solutions $u\in L^{1,2}_{\text{loc}}(\Omega)$ of (\ref{schrointro}) satisfying an additional \textit{logarithmic Caccioppoli-type condition}, and the form boundedness of the potential $\sigma$; see Theorem \ref{gensymthm}.

As a consequence, necessary
and sufficient conditions  will be established for both the form-boundedness
and  the positivity
of the quadratic form of the Schr\"{o}dinger type operator
$\mathcal{H}=-\text{div} (\mathcal{A}\nabla \cdot ) - \sigma$ with arbitrary distributional
potential  $\sigma \in D'(\Omega)$. The form
boundedness property (\ref{imbintro}) is known to be equivalent to the boundedness of the operator $\mathcal{H}: \Lo(\Omega) \to L^{-1,2}(\Omega) $ from the homogeneous Sobolev space $\Lo(\Omega)$ into its dual.  It is therefore a natural class of potentials in which to study the Schr\"{o}dinger equation.  In a wide class of domains $\Omega$, it has been characterized by the second and third authors \cite{MV1, MV5}.

Our results for the equations (\ref{schrointro}) and (\ref{ricintro}) in turn provide an alternative proof  (with a sharp upper
form bound) of the characterization of (\ref{imbintro}) established in \cite{MV1, MV5}, where harmonic analysis and potential theory methods were employed.  In addition, we obtain a characterization of potentials $\sigma\in \mathcal{D}'(\Omega)$ satisfying the corresponding \textit{semi-boundedness}
property, so that  the operator $\mathcal{H}$ is non-negative:
\begin{equation}\label{halfform}
\langle \sigma, h^2 \rangle \leq  \int_{\Omega} (\mathcal{A}\nabla h )\cdot \nabla h \, dx, \text{ for all }h\in C^{\infty}_0(\Omega).
\end{equation}

Both equations (\ref{schrointro}) and (\ref{ricintro}), as well as  the
quadratic form properties of  $\mathcal{H}$ (\ref{imbintro}) and (\ref{halfform})
 are of fundamental importance to partial differential equations, spectral theory,  and mathematical physics. Consequently,  these questions  have attracted the attention of many authors,
starting from the foundational work of B\^{o}cher, Hartman, Hille, and Wintner   on the Sturm-Liouville theory (see e.g., \cite{Hi48}, \cite{Har82}, Chapter 11),  followed by contributions of Agmon \cite{Ag83}, Aizenman and Simon \cite{AS82}, Ancona \cite{An86},
Brezis and Kato \cite{BK79}, Chung and Zhao \cite{CZ95}, Maz'ya
\cite{Maz85},
Murata \cite{Mur86} et al. in the multi-dimensional case. A recent survey of this rich area has been given by Pinchover \cite{Pin07}. We  also refer to
 \cite{Maz69, CFKS89, NP92, BNV94, BM97, RS98, RSS94, SW99, Fit00, MS00, Sha00, Mur02, DN02, DD03}   and references therein for equation (\ref{schrointro}) and
form-boundedness properties (\ref{imbintro}), (\ref{halfform}), and  \cite{AHBV, Ev90, FM98, FM00} for equation (\ref{ricintro}).

 Given the wealth of the previous literature, it is important to stress what is novel about our approach.  In all the papers
 listed above, various assumptions on the potential $\sigma$ ensure
 the validity of Harnack's inequality for positive solutions of the
 Schr\"odinger equation or some form of compactness properties
 of $\mathcal{H}$. Moreover, $\sigma$ is usually decomposed into the sum of its
 positive and negative parts: $\sigma= \sigma_{+}-\sigma_{-}$, which are treated separately, with more stringent assumptions on $\sigma_{+}$
 than $\sigma_{-}$. In many of these results $\sigma_+$ is assumed to
 belong to the Kato class of potentials, while $\sigma_{-}$ to the local Kato class (see \cite{AS82}, \cite{CZ95}). The corresponding
 positive solutions are continuous, and the existence of a positive
 solution is equivalent to the positivity of the quadratic form of $\mathcal{H}$. In the mathematical physics literature the latter  is known as  the Allegretto-Piepenbrink theorem (see e.g. \cite{CFKS89}, Sec. 2). All these tools are not available for general potentials $\sigma$.

   The primary
technical hurdles of our approach in comparison with the existing literature arise from the following essential characteristics of $\sigma$ satisfying (\ref{imbintro}):
\begin{enumerate}
\item $\sigma$ in general does not lie \textit{globally} in a dual Sobolev space, i.e. $\sigma\not\in L^{-1,s}(\Omega)$ for any $s>0$;
\item there are no local compactness conditions on $\sigma$.
\end{enumerate}

From the first item above, it is clear that one cannot follow standard methods to achieve global estimates which would yield the existence of solutions of (\ref{schrointro}).  Indeed, there are simple examples of $\sigma$ so that a solution $u$ of (\ref{schrointro}) does not lie in $L^1(\Omega)$.  On the other hand, as a result of the second item, finding the correct quantity to work with in order to prove local estimates
becomes a subtle issue.

We will see that the two inequalities contained in (\ref{imbintro}) are responsible for two distinct aspects of the existence of solutions to (\ref{schrointro}) and (\ref{ricintro}).
Let us therefore consider the following upper and lower bounds of the quadratic form $\langle \sigma , h^2\rangle$:
\begin{equation}\label{introupper}
  \langle \sigma , h^2\rangle \leq \lambda \int_\Omega (\mathcal{A}\nabla h)
\cdot\nabla h \, dx, \quad \text{ for all }h\in C^{\infty}_0(\Omega),
\end{equation}
and
\begin{equation}\label{introlower}
  \langle \sigma , h^2\rangle  \geq -\Lambda \int_\Omega (\mathcal{A}\nabla h)\cdot
\nabla h \, dx, \quad \text{ for all }h\in C^{\infty}_0(\Omega).
  \end{equation}
In what follows, a positive function $u$ on $\Omega$ is defined to be a function $u\in L^{1,2}_{\text{loc}}(\Omega)$ such that $u>0$ quasi-everywhere in $\Omega$. Let us now state our first theorem:
  \begin{thm}\label{gensymthm}  Let $\Omega \subseteq \mathbf{R}^{n}$, $n \ge 1$, be an open set. Let $
\sigma \in D'(\Omega)$, and $
\mathcal{A}:\Omega \rightarrow \mathbf{R}^{n\times n}$ be a
matrix function satisfying the ellipticity conditions (\ref{elliptic}).  Then the following
statements hold:

  \indent {\rm(i)} Suppose $\sigma$ obeys
(\ref{introupper}) with an upper form bound $\lambda < 1$, and (\ref{introlower}) with a lower form bound $\Lambda>0$.  Then $\sigma\in L^{-1,2}_{\text{loc}}(\Omega)$, and
there exists a positive solution $u\in L^{1,2}_{\text{loc}}(\Omega)$ of
the equation
  \begin{equation}\label{schro-equation}
  -{\rm{div}}(\mathcal{A}\nabla u) = \sigma \, u \quad \text{ in }D'(\Omega),
    \end{equation}
    so that the following \textit{logarithmic Caccioppoli inequality} holds:
       \begin{equation}\label{class1}
\int_{\Omega}  \frac {|\nabla  u|^2}{u^2} \varphi^2 dx \leq C_0\int_{\Omega} |\nabla
\varphi|^2 dx, \quad \text{ for all }\varphi\in C^{\infty}_0(\Omega),
   \end{equation}
with $C_0=C_0(n, m, M, \Lambda)>0$.

  \indent {\rm(ii)} Suppose $\sigma\in L^{-1,2}_{\text{loc}}(\Omega)$, and there exists a solution $u\in L^{1,2}_{\text{loc}}(\Omega)$ satisfying (\ref{class1}) for a constant $C_0>0$.  Then there exists a solution $v\in
L^{1,2}_{\text{loc}}(\Omega)$ of the equation
   \begin{equation} \label{riccati-equation}
  -{\rm{div}}(\mathcal{A}\nabla v) = \mathcal{A}(\nabla v)\cdot
\nabla v + \sigma \quad \text{ in }D'(\Omega),
    \end{equation}
    such that
     \begin{equation}\label{class2}
\int_{\Omega}  |\nabla  v|^2 \, \varphi^2 dx \leq C_0\int_{\Omega} |\nabla
\varphi|^2 dx, \quad \text{ for all }\varphi\in C^{\infty}_0(\Omega).
   \end{equation}

  \indent {\rm(iii)} Suppose (\ref{riccati-equation}) has a solution $v\in
L^{1,2}_{\text{loc}}(\Omega)$ satisfying (\ref{class2}) for a positive constant $C_0>0$.  Then
  $\sigma$ satisfies  the lower form bound (\ref{introlower}) for a positive constant $\Lambda = \Lambda(C_0, m, M)>0 $, and
the upper form bound (\ref{introupper}) with

\indent\indent\indent(a) $\, \lambda =1$ if $\mathcal{A}$ is symmetric;

\indent\indent\indent(b) $\, \displaystyle \lambda = \Bigl(\frac{M}{m}\Bigl)^2$ if $\mathcal{A}$ is non-symmetric.

      \end{thm}


In the case where $\sigma$ is a \textit{positive measure}, the relationship between positive \textit{superharmonic} supersolutions of (\ref{schro-equation}) and the validity of (\ref{introupper}) has been discussed in a general framework using probabilistic methods by Fitzsimmons in \cite{Fit00}.  This work was in turn building on the important paper of Ancona \cite{An86}.  In our framework, as we are considering oscillating potentials, one cannot rely on the theory of superharmonic functions, and we need to prove sharper estimates in order to obtain the stronger $L^{1,2}_{\text{loc}}$ regularity.  Without this it is not obvious how to even make sense of solutions to the equation (\ref{schrointro}).

The sharpness of Theorem \ref{gensymthm} is exhibited by well known examples.  One such example is included in Sec.~\ref{examples} for the convenience of the reader.  Here it is also shown that in general there exist positive solutions $u$ of (\ref{schrointro}) which lie in the space $L^{1,1}_{\text{loc}}(\Omega)$, but not $L^{1,2}_{\text{loc}}(\Omega)$, and that statement (i) of Theorem \ref{gensymthm} fails in general if $\lambda = 1$.  One should also note that in statement (iii), in the case of a non-symmetric matrix $\mathcal{A}$, the constant $\lambda$ must in general depend on $M$ and $m$, see Sec.~\ref{examples}.

\subsection{} In Theorem \ref{gensymthm}, it was seen that the lower form bound (\ref{introlower}) on the potential $\sigma$ is necessary in order to obtain solutions satisfying the regularity conditions (\ref{class1}) and (\ref{class2}).  These conditions are of importance in our application of Theorem \ref{gensymthm} in characterizing the inequality (\ref{imbintro}), and are also of classical interest in partial differential equations.

However, if one is solely interested in the existence of solutions to (\ref{schrointro}) and (\ref{ricintro}), then the conditions on both the operator $\mathcal{A}$ and $\sigma$ may be relaxed to local conditions.  We say that $\mathcal{A}$ satisfies a \textit{local ellipticity and boundedness} assumption if, for each $U\subset\subset \Omega$, there exist positive constants $m_U$ and $M_U$, such that for any $x\in U$, we have
\begin{equation}\label{locelliptic}
m_U|\xi|^2 \leq \mathcal{A}(x)\xi \cdot \xi, \text{ and }|\mathcal{A}(x)\xi|\leq M_U|\xi|,\text{ for all }\xi\in \mathbf{R}^n.
\end{equation}
The condition (\ref{introlower}) can be relaxed to a local condition stated in terms of dual Sobolev spaces: when $n\geq 3$, suppose
\begin{equation}\label{sigmalocmor}\sigma = \text{div}(\vec G)  \text{ in }\Omega, \text{ where } \vec G\in \mathcal{L}^{2,n-2}_{\text{loc}}(\Omega)^n.
\end{equation}
Here $\mathcal{L}_{\text{loc}}^{2, n-2}(\Omega)$ is the local Morrey space defined in Sec.~\ref{Preliminaries}; see (\ref{locmordef})
and  (\ref{locsobball})   below.  This condition is significantly weaker than (\ref{imbintro}). In dimensions $n=1, 2$ we only require $\sigma \in L^{-1,2}_{\text{loc}}(\Omega)$.

The following theorem should be compared to statements (i) and (ii) of Theorem \ref{gensymthm} above:
\begin{thm}\label{refine}Let $\Omega \subseteq \mathbf{R}^{n}$ be an open set, and suppose that $
\mathcal{A}$ satisfies the local ellipticity and boundedness conditions (\ref{locelliptic}).  Let $\sigma\in \mathcal{D}'(\Omega)$ satisfy (\ref{introupper}) with a constant $0<\lambda<1$, and in addition suppose\\
\indent (a) $\, n=1, 2$, and $\sigma\in L^{-1,2}_{\text{loc}}(\Omega)$;\\
\indent (b) $\, n\geq 3$, and $\sigma$ satisfies the local condition
(\ref{sigmalocmor}).\\
Then there exists a positive solution $u\in L^{1,2}_{\text{loc}}(\Omega)$ of (\ref{schrointro}), and a solution $v\in L^{1,2}_{\text{loc}}(\Omega)$ of (\ref{ricintro}).
\end{thm}

 A crude sufficient condition for (\ref{sigmalocmor}) is $\sigma \in L^{n/2}_{\text{loc}}(\Omega) +L^{-1,n}_{\text{loc}}(\Omega)$, but much more general $\sigma$ are admissible for (\ref{sigmalocmor}).
 We emphasize that condition (\ref{sigmalocmor}) is considerably weaker than the usual  local Kato class condition. It  is not necessary for the existence of a positive solution $u\in L^{1,2}_{\text{loc}}(\Omega)$ of (\ref{schrointro}).   However, it is the sharp condition to obtain solutions $u$ of (\ref{schrointro}) so that $\log(u) \in BMO_{\text{loc}}(\Omega)$.

To prove Theorems \ref{gensymthm} and \ref{refine}, we make crucial use of certain Caccioppoli-type inequalities. As was mentioned above, the classical iterative techniques used in \cite{Mos60, Tru73, BK79, AS82, CFG86, MZ97} are not available.

 Instead, we interpolate between a Caccioppoli inequality and an estimate on the mean oscillation of the logarithm to obtain uniform doubling properties on an approximating sequence. See  Proposition \ref{mainprop}, which
constitutes a key part of the argument.  From this doubling property, one can deduce local uniform gradient estimates.  This technique yields the optimal regularity for solutions of (\ref{schrointro}) in the generality of potentials satisfying (\ref{imbintro}) or (\ref{sigmalocmor}).

Along the way, we obtain a characterization of when a nonnegative weight function satisfying a weak reverse H\"{o}lder inequality is doubling
(see Sec.~2.2 for definitions).
Our main hard analysis tool here is  Proposition \ref{locdoub} which  may be of independent interest.

\subsection{} Let us now turn to discussing applications of Theorem \ref{gensymthm}.  As a first application, we deduce an alternative approach to the results of the second and third authors in \cite{MV1}, regarding the characterization of the inequality (\ref{imbintro}).  It  avoids the heavy harmonic analysis and potential theory  machinery that was used in \cite{MV1}, and is considerably more elementary.  This program is carried out in Sec.~\ref{formbdsec}.  In particular, if $\Omega = \mathbf{R}^n$ and with $\mathcal{A}$ the identity matrix, we will show that the form boundedness condition (\ref{imbintro}) is equivalent to the following representation of $\sigma$:
\begin{equation}\label{mazverintro}
\sigma = \text{div}(\vec\Gamma), \!\text{ with }\!\! \int_{\mathbf{R}^n}\!\!h^2 |\vec\Gamma|^2 dx\leq C_1 \int_{\mathbf{R}^n}\!\!|\nabla h|^2 dx \text{ for all }h \in C^{\infty}_0(\Omega).
\end{equation}
Moreover, (\ref{mazverintro}) with $C_1=\frac 1 4$ implies (\ref{imbintro}) with $C=1$. Conversely, (\ref{imbintro}) with $C<1$
(or more precisely, (\ref{introupper}) with $\lambda <1$ and  (\ref{introlower}) with $\Lambda>0$) implies (\ref{mazverintro}).

In Section \ref{sembdsec}, we consider distributions $\sigma\in \mathcal{D}'(\Omega)$ satisfying the semi-boundedness
property (\ref{halfform}). In the case of the Laplacian, it means
\begin{equation}\label{sembdintro}
\langle \sigma, h^2\rangle \leq \int_{\Omega} |\nabla h|^2 \, dx, \quad \text{ for all }h\in C^{\infty}_0(\Omega).
\end{equation}
Our main result in this regard is the following. (See Theorem~\ref{semithm} below for a similar criterion
concerning the general operator ${\rm div }\mathcal{A}(\nabla \cdot)$ in place of the Laplacian.)
\begin{thm}\label{critric} A real-valued distribution $\sigma\in \mathcal{D}'(\Omega)$ satisfies (\ref{sembdintro}) if and only if there exists $\vec\Gamma\in L^2_{\text{loc}}(\Omega)^n$, so that
\begin{equation}\label{semibdcondintro}\sigma \leq {\rm div }(\vec\Gamma) - |\vec\Gamma|^2 \qquad \text{ in }\mathcal{D}'(\Omega).
\end{equation}
\end{thm}
 The inequality in (\ref{semibdcondintro}) can not in general be strengthened to an equality.
Such conditions have their roots in classical Sturm-Liouville theory in
one dimension, see e.g.  \cite{Har82} (Chapter 11, Theorem 7.2).

It had been conjectured that a condition characterizing (\ref{sembdintro}) was the following:
\begin{equation}\label{falseintro}
\sigma \leq \text{div}(\vec\Phi), \text{ where } \int_{\Omega}|h|^2 |\vec\Phi|^2 dx\leq C  \int_{\Omega}|\nabla h|^2 dx,
\text{ for all }h \in C^{\infty}_0(\Omega),
\end{equation}
for some $\vec \Phi \in L^2_{\text{loc}}(\Omega)^n$ and $C>0$.
In other words, this means  that `half' the condition (\ref{mazverintro}) found in \cite{MV1} should characterize semi-boundedness.  However, it is proved below that, for any $C>0$, condition  (\ref{falseintro}) is \textit{not necessary} for   (\ref{sembdintro}) to hold, although it
is obviously sufficient when $C=\frac 1 4$.
\begin{prop}Let $\Omega = \mathbf{R}^n$, $n\geq 1$.   Let $\sigma$ be the radial potential defined by
$$ \sigma = \cos r + \frac{n-1}{r}\sin r -\sin ^2 r,
$$
where $r=|x|$. Then $\sigma$ satisfies (\ref{sembdintro}), but cannot be represented in the form (\ref{falseintro}).
\end{prop}
This is the content of Proposition~\ref{conjfalse} in Sec.~7, where additional examples are exhibited to help clarify our results.

Theorem \ref{critric} above concerned the case when $\lambda=1$ in (\ref{introupper}), so that the Schr\"{o}ding\-er operator fails to be coercive in the homogeneous Sobolev space.  In the case when $\sigma$ is a positive measure, one can instead consider superharmonic supersolutions of the equation (\ref{schrointro}) in this critical case.  This is a sharpening of \cite{Fit00} mentioned above.  Indeed:
\begin{prop}\label{critpropintro}  Suppose that $\Omega$ is an open set, and let $\sigma$ be a positive Borel measure defined on $\Omega$.  Let $\mathcal{A}$ be a symmetric matrix function satisfying (\ref{elliptic}).  Then $\sigma$ satisfies
\begin{equation}\label{positive1}
\int_{\Omega}h^2 d\sigma \leq \int_{\Omega} \mathcal{A}(\nabla h)\cdot\nabla h \, dx, \text{ for all } \, \,  h\in C^{\infty}_0(\Omega),
\end{equation}
if and only if there exists a positive superharmonic function $u$ so that
$$-{\rm div }\mathcal{A}(\nabla u) \geq \sigma u \quad \text{ in }\Omega.$$
\end{prop}
In Proposition \ref{critpropintro}, the notion of superharmonicity is the one associated to the operator $\mathcal{A}$, see e.g. \cite{HKM}.  This is proved in Section \ref{criticalcase} below.

In Section \ref{FNVsec}, we consider a recent result of Frazier, Nazarov and the third author \cite{FNV} on
positive solutions with prescribed boundary values.  Let us recall one of their main theorems.  Suppose $\Omega$ is a bounded NTA domain, and that $\sigma$ is a nonnegative measure in $\Omega$.  Then, under precise necessary and sufficient conditions on $\sigma$ up to the boundary,  a positive minimal solution $u$ (called the \textit{gauge} in the probabilistic literature, see e.g. \cite{CZ95}) to the  equation
\begin{equation}\label{veryweaksoln}\begin{cases}
-\Delta u = \sigma u \quad \text{ in }\Omega,\\
\,u=1 \quad  \text{ on }\partial\Omega,
\end{cases}\end{equation}
is constructed in \cite{FNV}   (see Theorem \ref{FNV} below).  The solution is understood in the sense that
$$u (x)  = \int_\Omega G(x, y) \,  u (y) \, d \sigma(y) +1,$$ where  $G(x,y)$ is Green's function
of the Laplacian.
In this paper, we will adapt the approach taken in the proof of Theorem \ref{gensymthm} to show that in fact $u \in L^{1,2}_{\text{loc}}(\Omega)$; see Theorem \ref{u1reg} below.  This regularity is again optimal under the assumptions  of the theorem.


In conclusion, we remark that our approach outlined above is nonlinear in nature, and an extension  to general quasilinear operators of $p$-Laplacian type will be presented in a forthcoming paper \cite{JMV10a}, where the $L^p$-analogue of (\ref{imbintro}) will be characterized.

\subsection{Acknowledgement}  We would like to thank Yehuda Pinchover for suggesting that 
Theorem \ref{refine} should require only \textit{local} ellipticity and boundedness conditions on the operator $\mathcal{A}$.

\section{Preliminaries}\label{Preliminaries}

\subsection{Notation and function spaces}\label{notation} For an open set
$\Omega \subseteq \mathbf{R}^{n}$, $n \ge 1$, we denote by $C^{\infty}_0(\Omega)$ the space of smooth functions with compact support in $\Omega$.  The energy space $\Lo(\Omega)$ is then the completion of $C^{\infty}_0(\Omega)$ with respect to the Dirichlet norm $\norm{\nabla h}_{L^2(\Omega)}$.  The majority of estimates in this paper are local; we say that $h \in \smash{L^{1,2}_{\text{loc}}}(\Omega)$ if $h\varphi\in \Lo(\Omega)$ whenever $\varphi \in C^{\infty}_0(\Omega)$.

For test function arguments it will be useful to introduce the space $L^{1,2}_c(\Omega)$.  We say that $h\in L^{1,2}_c(\Omega)$ if $h\in \Lo(\Omega)$ has compact support.

Define $L^{-1, 2}(\Omega)$ to be the dual of $\Lo(\Omega)$.  Then a distribution $\sigma \in L^{-1,2}_{\text{loc}}(\Omega)$ if $\varphi \sigma \in L^{-1,2}(\Omega)$ for any $\varphi\in C^{\infty}_0(\Omega)$.

We will write $V\subset\subset U$, for two open sets $U, V\subset \mathbf{R}^n$, if there exists a compact set $K\subset \mathbf{R}^n$ so that $V\subset K\subset U$.

Throughout the paper,  we use the usual notation for the integral average:
$$\dashint_E \cdots \, dx = \frac{1}{|E|}\int_E \cdots \,  dx.
$$
For an open set $U$, we say $u\in BMO(U)$ if there is a positive constant $D_U$ so that
\begin{equation}\label{BMOdef}\dashint_{B(x,r)} |u(y)-\dashint_{B(x,r)} u(z) \,dz|^2 dy \leq D_U, \text{ for any ball }B(x,2r)\subset U.
\end{equation}
In addition, $u\in BMO_{\text{loc}}(\Omega)$ if for each compactly supported open set $U\subset\subset\Omega$, there is a positive constant $D_U>0$ so that (\ref{BMOdef}) holds.

Let us next introduce the \textit{local Morrey space}: we say $f\in \mathcal{L}^{p,q}_{\text{loc}}(\Omega)$ if, for each compactly supported  set $U\subset\subset\Omega$, there exists a constant $C_U$ so that
\begin{equation}\label{locmordef}
\int_{B(z,s)}|f|^p \, dx\leq C_U s^{q}, \quad \text{ for all balls }B(z,2s)\subset U.
\end{equation}
We conclude with the definition of a multiplier (see \cite{MSh09}).  Let $X$ and $Y$ be two normed function spaces, and let $Z$ be a dense subset of $X$.  We say that $g$ is a multiplier from $X$ to $Y$, written as $g\in M(X\rightarrow Y)$, if $g\cdot f\in Y$ for all $f\in Z$, and there is a positive constant $C>0$ so that the following inequality holds:
$$\norm{g \cdot f}_Y\leq C \norm{f}_X, \quad \text{ for all }f\in Z.$$
In what follows $X$ and $Y$ will be $\Lo(\Omega)$ and $L^2(\Omega)$ respectively, and $Z$ will be $C^{\infty}_0(\Omega)$.

\subsection{On weak reverse H\"{o}lder inequalities and BMO}\label{wrhsec} In this section we characterize the weak reverse H\"{o}lder weights that are doubling.  This forms a key tool in our argument.  First, let us introduce some notation.

\begin{defn}  Let $U\subset\mathbf{R}^n$ be an open set, and let $w$ be a nonnegative measurable function.  Then $w$ is said to be \textit{doubling} in $U$ if there exists a constant $A_U>0$ so that,
\begin{equation}\label{ld}\dashint_{B(x,2r)} w\, dx \leq A_U\dashint_{B(x,r)} w \, dx, \text{ for all balls }B(x,4r)\subset U.
\end{equation}
Let $w$ be a nonnegative measurable function.  Then $w$ is said to satisfy a \textit{weak reverse H\"{o}lder inequality} in $U$ if there exists constants $q>1$ and $B_U>0$ so that,
\begin{equation}\label{wrh}\Bigl(\dashint_{B(x,r)} w^q dx\Bigl)^{1/q} \leq B_U\dashint_{B(x,2r)} w\,dx, \text{ for all balls }B(x,2r)\subset U.
\end{equation}
\end{defn}

\begin{rem}\label{chainarg}  The following simple consequence of the doubling property will prove useful.  Let $U$ be an open set, and suppose $w$ is doubling in $U$.  Then, whenever $B(x,4r)\subset U$ and $z\in B(x,r)$ with $B(z,4s)\subset U$, we have
$$\dashint_{B(x,r)} w(y) dy \leq C(A_U, s, r) \dashint_{B(z,s)} w(z) dz.
$$
This principle will be used in a Harnack chain argument in Proposition \ref{mainprop}.
\end{rem}

Our argument hinges on the following result:

\begin{prop} \label{locdoub} Let $U$ be an open set, and suppose $w$ satisfies the weak reverse H\"{o}lder inequality (\ref{wrh}) in $U$.  Then $w$ is doubling in $U$, i.e. (\ref{ld}) holds, if and only if $\log(w) \in BMO(U)$ (see (\ref{BMOdef})).

In particular, if $w$ satisfies (\ref{wrh}) and
\begin{equation}\label{bmo1}
\dashint_{B(x,s)} \!\!| \log w(y) - \dashint_{B(x,s)} \!\!\log w(z) dz|^2 dy \leq D_U,
\end{equation}
for all balls $B(x,2s)\subset U$.  Then there is a constant $C(B_U, D_U)>0$, so that for any ball $B(x,4r)\subset U$
\begin{equation}\label{doub}
\dashint_{B(x,2r)} w\, dx \leq C(B_U, D_U) \dashint_{B(x,r)} w\, dx,
\end{equation}
where $B_U$ is the constant from (\ref{wrh}).
\end{prop}

Only the the sufficiency direction is required in what follows; however, since this characterization does not seem to appear explicitly in the literature we prove the full statement.  To prove Proposition \ref{locdoub}, we use the following lemma:

\begin{lem} \label{reverseholder} Let $U\subset \mathbf{R}^n$ be an open set.  Suppose that there exist $s>1$ and $w\geq 0$, along with a constant $C_1>0$ so that the following inequality holds:
$$\Bigl(\dashint_{B(x,r)} w^{s} \, dx\Bigl)^{1/s} \leq C_1 \dashint_{B(x, 2r)} w\,dx, \text{ whenever } B(x, 2r)\subset U.
$$
Then, for any $t>0$, there exists a constant $C_t=C(t, C_1)>0$ so that
$$\Bigl(\dashint_{B(x,r)} w^{s} \, dx\Bigl)^{1/s} \leq C_t \Bigl(\dashint_{B(x, 2r)} w^t\,dx\Bigl)^{1/t}, \quad \text{ whenever } B(x, 2r)\subset U.
$$
\end{lem}
This lemma had been used in proving estimates for quasilinear equations by G. Mingione \cite{Min07}.  A proof can be found in Remark 6.12 of \cite{Giu03}.   Let us now turn to proving the proposition.

\begin{proof}[Proof of Proposition \ref{locdoub}]  Let us first prove the necessity, suppose that $w$ satisfies the weak reverse H\"{o}lder inequality (\ref{wrh}), and in addition that $w$ is doubling in $U$.  Then, for each ball $B(x,4r)\subset U$, we have
$$\Bigl(\dashint_{B(x,r)} w^q dx\Bigl)^{1/q}\leq B_U \dashint_{B(x,2r)} w\, dx \leq  A_UB_U \dashint_{B(x,r)} w \, dx.
$$
It follows that $w$ satisfies a reverse H\"{o}lder inequality in $U$, and is therefore a Muckenhoupt $A_{\infty}$-weight.  It follows (see Chapter 5 of \cite{St93}) that $\log(u)\in BMO(U)$.

Let us now turn to the converse statement.  Suppose $w$ satisfies (\ref{bmo1}) and (\ref{wrh}).  From (\ref{bmo1}), it is a well known consequence of the John-Nirenberg inequality that there exists a constant $0<t\leq1$ so that $w^t$ is an $A_2$-weight in $U$, i.e. there exists a positive constant $A>0$ (depending on $D_U$ in (\ref{bmo1})) so that for all balls $B(z,2s)\subset U$,
\begin{equation}\label{a2def}
\dashint_{B(z,s)} w^t \, dx \leq A \Bigl(\dashint_{B(z,s)} w^{-t} \, dx\Bigl)^{-1}.
\end{equation}
Indeed, the John-Nirenberg inequality (see \cite{St93}) yields $t=t(D_U)>0$ so that
\begin{equation}\label{jnineq}
\dashint_{B(z,s)} \exp\Bigl(t\Bigl|\log(w)(y') - \dashint_{B(z,s)} \log(w(y)) dy\Bigl|\Bigl)dy' \leq C(D_U).
\end{equation}
The inequality (\ref{jnineq}) contains two inequalities:
$$\dashint_{B(z,s)}\exp\Bigl(\log(w^t(y')) - \dashint_{B(z,s)} \log(w^t(y)) dy\Bigl) dy' \leq C(D_U), \text{ and}
$$
$$\dashint_{B(z,s)}\exp\Bigl( \log(w^{-t}(y')) + \dashint_{B(z,s)} \log(w^t(y)) dy\Bigl) dy' \leq C(D_U).
$$
Multiplying these two inequalities together, one obtains (\ref{a2def}).

Combining (\ref{a2def}) with Jensen's inequality, we see that if $B(z,4s)\subset U$, then
\begin{equation}\label{a2doub}
\dashint_{B(z,2s)} w^t \, dx \leq A2^n (\dashint_{B(z,s)} w^{-t} \,  dx\Bigl)^{-1}\leq A2^n \dashint_{B(z,s)} w^t \, dx.
\end{equation}
Let $B(z,8s)\subset U$, then applying Lemma \ref{reverseholder} yields
\begin{equation}\begin{split}\nonumber\dashint_{B(z,2s)} w\, dx\leq C_{U,t} \Bigl(\dashint_{B(z,4s)} w^t \, dx\Bigl)^{1/t} &\leq \tilde{C}_{U,t}\Bigl(\dashint_{B(z,s)} w^t \, dx\Bigl)^{1/t} \\
&\leq \tilde{C}_{U,t}\dashint_{B(z,s)} w\, dx.
\end{split}\end{equation}
The second inequality in the chain follows from the doubling of $w^t$, and the last inequality follows from H\"{o}lder's inequality.  By a standard covering argument, the factor of $8$ in the enlargement of the ball can be replaced by $4$, which yields (\ref{doub}).  This completes the proposition.
\end{proof}

\subsection{Preliminaries for distributional potentials $\sigma$}\label{prelimdist}
Let $\Omega$ be an open set in $\mathbf{R}^n$, with $n\geq 1$.  Let $\mathcal{A}: \Omega \rightarrow \mathbf{R}^{n\times n}$ satisfying the \textit{local conditions} (\ref{locelliptic}).  For a real-valued distribution $\sigma$ defined on $\Omega$, we define the multiplication operator by
$$ \langle \sigma h, h\rangle :=  \langle \sigma , h^2\rangle, \text{ for all } h\in C^{\infty}_0(\Omega).
$$
Suppose now $\sigma$ satisfies (\ref{imbintro}) for a positive constant $C>0$, and let us write:
$$||h||_{\mathcal{A}} = \Bigl(\int (\mathcal{A}\nabla h)\cdot\nabla h dx\Bigl)^{1/2}, \text{ for } h\in C^{\infty}_0(\Omega) $$
By polarization, we see that (\ref{imbintro}) is equivalent to the inequality:
\begin{equation}\label{sesqform}
|\langle \sigma g, h\rangle|\leq C||g||_{\mathcal{A}}||h||_{\mathcal{A}} \text{ for all } g,h \in C^{\infty}_0(\Omega),
\end{equation}
with $C>0$ the same constant that appears in (\ref{imbintro}).  Furthermore, by the local boundedness assumption (\ref{locelliptic}) on the operator $\mathcal{A}$, we can extend (\ref{sesqform}) by continuity, so that (\ref{sesqform}) is valid for all $g,h \in \Lo(U)$ whenever $U\subset\subset \Omega$.  We denote this extension again by $\sigma$.

We now state some simple lemmas regarding the local character of the distributional potentials we will consider.  Let us begin with an alternative way of stating the local condition (\ref{sigmalocmor}):

\begin{lem}\label{locmorlocsobsame}
Suppose that $n\geq 3$ and $\sigma \in \mathcal{D}'(\Omega)$. If $\sigma$ satisfies (\ref{sigmalocmor}) then the following condition holds:
\begin{equation}\label{locsobball}\begin{split}&\text{For each open set }U\subset\subset \Omega,\text{ there exists a constant }C_U>0,\\
&\text{such that whenever } B(z,2s)\subset U:\\
&|\langle h, \sigma\rangle |\leq C_U s^{(n-2)/2}||\nabla h||_{L^2(B(z,s))}, \text{ for all }h\in C^{\infty}_0(B(z,s)).
\end{split}\end{equation}
\end{lem}

The two conditions (\ref{locsobball}) and (\ref{sigmalocmor}) are in fact equivalent, but we omit the proof of the converse statement to Lemma \ref{locmorlocsobsame}, as we will not use it, and its proof is slightly lengthy.  To prove the lemma, note that if $\sigma = \text{div}(\vec G)$, with $|\vec G|$ satisfying (\ref{sigmalocmor}), then it follows directly from differentiation that (\ref{locsobball}) holds.

We next prove a key local property of distributions $\sigma$ satisfying (\ref{imbintro}), namely that (\ref{imbintro}) is a stronger condition than the condition (\ref{locsobball}).  We define the capacity $\text{cap}(E, \Omega)$ of a compact set $E$ by:
\begin{equation}\label{capdef}
\text{cap}(E, \Omega ) = \inf \{||\nabla h||^2_{L^2(\Omega)}\, :\, h\in C^{\infty}_0(\Omega), \, h\geq 1\text{ on }E\}.
\end{equation}
\begin{lem}\label{localsoblem}
Suppose that $\sigma$ satisfies (\ref{imbintro}), and let $U\subset\subset V\subset\subset \Omega$.  Then $\sigma \in L^{-1,2}(U)$, and
\begin{equation}\label{locadualest}
||\sigma||_{L^{-1,2}(U)} \leq M_V C\cdot{\rm{cap}}(U, V)^{1/2}.
\end{equation}
Here $M_V$ is as in (\ref{locelliptic}), and $C$ is the constant from (\ref{imbintro}).  In particular, for each ball $B(x,r)$ so that $B(x,2r)\subset U$, we have
\begin{equation}\label{embimpliesloc}||\sigma||_{L^{-1,2}(B(x,r))} \leq C_1(C, M_U)r^{(n-2)/2}.
\end{equation}
\end{lem}

From display (\ref{embimpliesloc}), it follows that (\ref{imbintro}) is stronger than the local condition (\ref{locsobball}).

\begin{proof}  Let $h\in C^{\infty}_0(U)$, and let $g\in C^{\infty}_0(V)$, so that $g\equiv 1$ on $U$.  Then applying (\ref{sesqform}) and (\ref{imbintro}), we obtain
$$|\langle \sigma , h\rangle| = |\langle \sigma g, h\rangle| \leq C ||g||_{\mathcal{A}}||h||_{\mathcal{A}} \leq M_VC ||\nabla g||_2||\nabla h||_2.
$$
Therefore $\sigma\in L^{-1,2}(U)$, and minimising over such $g$ yields (\ref{locadualest}) by definition of capacity.  The second estimate is a special case of (\ref{locadualest}) and follows from well known estimates for the capacity of a ball, see for example \cite{Maz85}.
\end{proof}

It will be convenient to use a mollification of the potential $\sigma$. Let us fix a smooth radial approximate identity $\varphi$, i.e. $\varphi\in C^{\infty}_0(B(0,1))$, so that $\varphi \geq 0$ on $B_{1}(0)$, with $||\varphi||_{L^1} =1$.  For $\varepsilon>0$, we denote $\varphi_{\varepsilon} = \varepsilon^{-n}\varphi(x/\varepsilon)$.  Then, denote the convolution of the distribution by $\sigma_{\varepsilon} = \varphi_{\varepsilon}*\sigma$.  The next two lemmas show that the mollification does not effect $\sigma$ in terms of form boundedness.  We will write $d\sigma_{\varepsilon} = \sigma_{\varepsilon} dx$.

\begin{lem}\label{mollem}  Let $\Omega$ be an open set, and let $V\subset\subset \Omega$.  In addition let $\varepsilon \leq d(V, \partial\Omega)/2$.  Suppose that $\sigma \in \mathcal{D}'(\Omega)$ so that (\ref{introupper}) holds for a constant $\lambda>0$.  Then the following inequality holds:
\begin{equation}\label{molconc}\int_{V} h^2 d\sigma_{\varepsilon} \leq \lambda\int_{V}(\mathcal{A}_{\varepsilon}\nabla h)\cdot \nabla h dx, \text{ for all }h\in C^{\infty}_0(V).
\end{equation}
where $\mathcal{A}_{\varepsilon}(x) = (\varphi_{\varepsilon}*\mathcal{A})(x).$
\end{lem}

\begin{proof}
Let $h\in C^{\infty}_0(V)$.  We first note that by the interchange of mollification and the distribution (see Lemma 6.8 of \cite{LL01})
$$\langle \sigma, \varphi_{\varepsilon}*h^2\rangle= \int_{B(0,\varepsilon)}\varphi_{\varepsilon}(t) \langle\sigma, h(\cdot-t)^2 \rangle dt.
$$
By elementary geometry, $h(\,\cdot-t)\in C^{\infty}_0(\Omega)$ for all $t\in B(0, \varepsilon)$, and hence
\begin{equation}\begin{split}\langle\sigma_{\varepsilon}, h^2\rangle & \leq \lambda \int_{B(0,\varepsilon)} \varphi_{\varepsilon}(t)\Bigl(\int_{\Omega}(\mathcal{A}(x)\nabla h(x-t))\cdot \nabla h(x-t) dx\Bigl)dt \\
& = \int_{\Omega}\mathcal{A}_{\varepsilon}(x)\nabla h(x)\cdot \nabla h(x) dx,
\end{split}\end{equation}
which proves the lemma.
\end{proof}

Our second mollification lemma says that if $\sigma$ satisfies the local condition (\ref{locsobball}), then so does the mollification of $\sigma$.  Let us introduce the notation $U_{\varepsilon} = \{x\in \mathbf{R}^n\, :\, \text{dist}(x,U)<\varepsilon\}$.

\begin{lem}\label{molball}  Suppose $n\geq 3$ and $\sigma \in \mathcal{D}'(\Omega)$ is a real valued distribution satisfying (\ref{locsobball}).  Let $V\subset\subset \Omega$, then if $\varepsilon < d(V, \partial\Omega/2)$, the mollified potential $\sigma_{\varepsilon}$ satisfies (\ref{locsobball}) for all open sets $U\subset\subset V$, with constant $C_{U_{\varepsilon}}$.
\end{lem}

\begin{proof}  Let $U\subset\subset V$ be a compactly supported open set.  Then, for $B(x,2r)\subset U$ and $h\in C^{\infty}_0(B(x,r))$, $h(\cdot-x)\in C^{\infty}_0(\Omega)$ for all $x\in B_{\varepsilon}(0)$.   Hence, for all such $h$, it follows that
\begin{equation}\begin{split}\Bigl|\int_{\Omega} h&  d\sigma_{\varepsilon}\Bigl|  \leq  \int_{B_{\varepsilon}(0)} \varphi_{\varepsilon}(x)|\langle \sigma, h(\cdot - x)\rangle| dx \\
&\leq C_{U_{\varepsilon}}r^{(n-2)/2}\int_{B_{\varepsilon}(0)} \varphi_{\varepsilon}(x)||\nabla h(\cdot - x)||_{L^2} dx\\
& \leq C_{U_{\varepsilon}}r^{(n-2)/2}||\nabla h||_{L^{2}}.
\end{split}\end{equation}
This completes the proof.
\end{proof}

\section{The proofs of Theorems \ref{gensymthm} and \ref{refine}}\label{linear}
In this section we prove our primary existence theorems, as well as prove the connections between the solutions of the equations (\ref{schrointro}) and (\ref{ricintro}) with the validity of (\ref{imbintro}).  Let $n\geq 1$. Throughout this section we will assume without loss of generality $\Omega\subset \mathbf{R}^n$ is a \textit{connected} open set; note that in the case of an arbitrary open set, our arguments apply to each connected component.  This assumption is used in a Harnack chain argument.

The most substantial argument will be the assertion of statement (i) in Theorem \ref{gensymthm}, and its local variant in Theorem \ref{refine}.  We restate these two results as propositions for convenience.  In light of Lemma \ref{locmorlocsobsame}, the existence result for equation (\ref{schrointro}) in Theorem \ref{refine} follows from:
\begin{prop}\label{refineexist}
Let $\mathcal{A}:\Omega\rightarrow\mathbf{R}^{n\times n}$ be a possibly non-symetric matrix function satisfying (\ref{locelliptic}). Suppose that $\sigma$ is a real-valued distribution satisfying the local dual Sobolev condition (\ref{locsobball}), and the upper boundedness condition (\ref{introupper}) for a constant $0<\lambda<1$.  Then there is a positive solution $u\in \smash{L^{1,2}_{\text{loc}}}(\Omega)$ of (\ref{schrointro}).
\end{prop}

The second proposition concerns the case when $\mathcal{A}$ satisfies \textit{global ellipticity and boundedness}, and $\sigma$ in addition satisfies (\ref{introlower}):

\begin{prop}\label{linexist}  Suppose $\mathcal{A}:\Omega\rightarrow\mathbf{R}^{n\times n}$ is a possibly non-symetric matrix function satisfying (\ref{elliptic}). In addition, suppose that $\sigma$ is a real-valued distribution satisfying (\ref{introlower}) for a positive constant $\Lambda>0$, and (\ref{introupper}) for a constant $0<\lambda<1$.  Then there is a positive solution $u\in \smash{L^{1,2}_{\text{loc}}}(\Omega)$ of (\ref{schrointro}).  Furthermore, $u$ satisfies
\begin{equation}\label{logmult}\int_{\Omega} \frac{|\nabla u|^2}{u^2} \varphi^2 \, dx \leq C\int_{\Omega} |\nabla \varphi|^2  \, dx \quad \text{ for all }\varphi\in C^{\infty}_0(\Omega).
\end{equation}
\end{prop}

It was seen in Lemma \ref{localsoblem} that if $\sigma$ satisfies (\ref{introupper}) and (\ref{introlower}), then the condition (\ref{locsobball}) holds.  Hence the existence part of Proposition \ref{linexist} follows from Proposition \ref{refineexist}.

\subsection{An approximating sequence} \label{approxlinesub}  To prove Propositions \ref{refineexist} and \ref{linexist}, we use local properties of $\sigma$ to find solutions to a mollified variant of equation (\ref{schrointro}) in a sequence of subdomains  of $\Omega$.  We will then prove a uniform gradient estimate on this sequence.

Suppose that $\sigma$ satisfies (\ref{introupper}) with $0<\lambda<1$.  Let $\Omega_j$, for $j\geq 1$ be an exhaustion of $\Omega$ by smooth domains, that is, $\Omega_j \subset\subset\Omega_{j+1}$, and $\Omega  = \bigcup_j \Omega_j $.  In addition, let us fix a ball $B$ so that its concentric enlargement $4B\subset\subset \Omega_1$.  Let $\varepsilon_0 = 1$, and $\varepsilon_j  = \min (\varepsilon_{j-1}/2, d(\Omega_j, \partial\Omega_{j+1})/2, 2^{-j})$ for $j\geq 1$.  With this notation, define:
$$\sigma_j = \varphi_{\varepsilon_j} * \sigma, \text{ and }\mathcal{A}_j = \varphi_{\varepsilon_j}*\mathcal{A}.
$$
Referring to (\ref{locelliptic}), we let $m_{j} = m_{\Omega_{j}}$ and $M_j = M_{\Omega_j}$ be the local ellipticity and boundedness constants relative to $\Omega_{j}$.  For $k\geq j$, note that $\varepsilon_k <d(\Omega_j, \partial\Omega_{j+1})$, and hence for any $x\in \Omega_j$ we have \begin{equation}\label{locellipticj}m_{j+1}|\xi|^2\leq \mathcal{A}_k(x)\xi \cdot \xi, \text{ and }|\mathcal{A}_k(x)\xi|\leq M_{j+1}|\xi|,\text{ for all }\xi\in \mathbf{R}^n.\end{equation}
Define $u_j$ to be the solution of
\begin{equation}\begin{cases}\label{approxlinear}
-{\rm{div}}(\mathcal{A}_j\nabla u_j) = \sigma_j u_j \quad  \text{ in } \Omega_j,\\
\displaystyle \int_B u_j^2 dx = 1,\quad u_j \geq 0 \, \text{ q.e.}
\end{cases}
\end{equation}
Furthermore, $u_j$ satisfies the Harnack inequality in $\Omega_j$.

\begin{proof}[Proof of existence and uniqueness of (\ref{approxlinear})]  We will see that the existence of (\ref{approxlinear}) is a simple consequence of the Lax-Milgram lemma.  Define a bilinear form $\mathcal{L}$ on $\Lo(\Omega_j)\times\Lo(\Omega_j)$ by
$$\mathcal{L}(w,h) = \int_{\Omega_j} (\mathcal{A}_j\nabla w) \cdot \nabla h dx - \langle \sigma_j w, h \rangle, \text{ with } w,h\in \Lo(\Omega_j).
$$
By the assumptions on $\sigma$, the following properties hold:
$$|\mathcal{L}(w,h)| \leq (M_{j+1}+C(n)||\sigma_j||_{L^{\infty}(\Omega_j)}|\Omega_j|^{2/n}) \norm{\nabla w}_2\norm{\nabla h}_2,\text{ and}
$$
$$\mathcal{L}(w,w) \geq m_{j+1}(1-\lambda) \norm{\nabla w}_2.
$$
The first inequality follows from (\ref{locellipticj}) along with the Sobolev inequality (this is true for all $n\geq 1$, using standard Sobolev inequalities). The second inequality is a combination of Lemma \ref{mollem} and (\ref{locellipticj}).  Hence the hypotheses of the Lax-Milgram lemma are satisfied.

Applying the Lax-Milgram lemma, we see that there exists a unique $w_j\in \Lo(\Omega_j)$ satisfying
$$\mathcal{L}(w_j, h) =\langle \sigma_j, h \rangle.
$$
Let $v_j = w_j+1$.  Let us next show that $v_j\geq 0$ q.e.  To see this, let $\varphi=v_j$, and note that $\varphi^- = \min (v_j,0) \in \Lo(\Omega_j)$.  By testing (\ref{approxlinear}) with the valid test function $\varphi^{-}$, it follows:
\begin{equation}\label{negativezero}\int_{\Omega_j} \mathcal{A}_j(\nabla \varphi^-)\cdot \nabla \varphi^- = \langle \sigma_j \varphi^-, \varphi^- \rangle  \leq \lambda \int_{\Omega_j} \mathcal{A}_j(\nabla \varphi^-)\cdot \nabla \varphi^- .
\end{equation}
Since $0<\lambda<1$, it follows from (\ref{elliptic}) that $\varphi^- = 0$ q.e as required.  Let us now define
$$u_j = \Bigl(\dashint_B v_j dx \Bigl)^{-1} v_j.
$$
Then $u_j$ solves (\ref{approxlinear}).  The validity of Harnack's inequality for $u_j$ follows from classical elliptic regularity theory, see e.g. \cite{Tru73}, since $\sigma_j$ is smooth.
 \end{proof}

\subsection{Caccioppoli and Morrey estimates for the approximating sequence}We next turn to proving two estimates on the gradient of the approximating sequence. The first estimate is a Caccioppoli inequality:
\begin{lem}\label{lemcacc}Suppose that $\sigma$ satisfies (\ref{introupper}) with $0<\lambda<1$, and let $\{u_j\}$ be the sequence constructed in (\ref{approxlinear}).  Let $\psi \in C^{\infty}_0(\Omega_j)$, then for any $k\geq j$,
\begin{equation}\label{caccio}
\int_{\Omega_j} |\nabla u_k|^2 \, \psi^2 \, dx \leq C(M_{j+1}, m_{j+1}, \lambda)\int_{\Omega_j} u_k^2 \,  |\nabla \psi|^2 \, dx.
\end{equation}
\end{lem}

\begin{proof}   Let us fix $k$ and $j$ as in the statement of the lemma, and let $v=u_k$.  With $\psi \in C^{\infty}_0(\Omega_j)$, $\psi\geq 0$, test the weak formulation of (\ref{approxlinear}) with $v \psi^2 \in \Lo(\Omega _j)$.  Using (\ref{introupper}), it follows that
\begin{equation}\nonumber\begin{split}
\int_{\Omega_j}( &(\mathcal{A}_k \nabla v) \cdot \nabla v) \psi^2 \, dx \leq \langle\sigma_k v, \psi^2 v\rangle + 2M_{j+1}\int_{\Omega_j} v \, \psi \,  |\nabla v| \, |\nabla \psi| \, dx\\
& =\langle\sigma_k ( \psi v), \psi v\rangle + 2M_{j+1}\int_{\Omega_j} v \, \psi \,  |\nabla v| \, |\nabla \psi \,| \, dx\\
& \leq \lambda \int_{\Omega_j} \mathcal{A}_k(\nabla(v\psi))(\nabla(v \, \psi)) \, dx + 2M_{j+1}\int_{\Omega_j} v \, \psi \, | \nabla v| \, |\nabla \psi| \, dx.
\end{split}\end{equation}
Here we have used Lemma \ref{mollem} and $\langle\sigma_k v, \psi^2 v\rangle = \langle \sigma_k (v\psi), v\psi\rangle$.  Using the Cauchy inequality, it follows that for any $\varepsilon>0$ there exists a constant $C_{\varepsilon}$, depending on $\varepsilon, \lambda$, $M_{j+1}$ and $m_{j+1}$, such that
$$(1-\lambda)\int_{\Omega_j} (\mathcal{A}_k \nabla v) \cdot \nabla v \, \psi^2 \, dx \leq \varepsilon \int_{\Omega_j} |\nabla v|^2 \psi^2 \, dx + C_{\varepsilon}\int_{\Omega_j} v^2 |\nabla \psi|^2 \, dx.
$$
Choosing $\varepsilon <(1-\lambda)m_{j+1}$ and rearranging, we recover (\ref{caccio}).
\end{proof}
The second estimate we use relates a bound on the gradient of the logarithm independent of $j$, with uniform properties on the negative part of the quadratic form associated with the distribution $\sigma$:
\begin{lem}\label{lemlog}  With $(u_j)_j$ the sequence constructed in (\ref{approxlinear}), the following estimate holds for all $k\geq j$ and $\psi\in C^{\infty}_0(\Omega_j)$,
\begin{equation}\begin{split}\label{gradlogest}
\int_{\Omega_j} \frac{|\nabla u_k|^2}{u_k^2} \psi^2 \, dx \leq &-C(M_{j+1},m_{j+1}) \int_{\Omega_j} \psi^2  \, d\sigma_k \\
&+ C(M_{j+1}, m_{j+1})\int_{\Omega_j} |\nabla \psi|^2 \, dx. \end{split}\end{equation}
\end{lem}

\begin{proof}  Let $h = \psi^2/u_k$, with $\psi\in C^{\infty}_0(\Omega_j)$, $\psi\geq 0$. Since $u_k$ satisfies the Harnack inequality in $\Omega_j$, there exists a constant $c>0$ so that $u_k >c$ on the support of $\psi$.  It follows that $h$ is a valid test function for the weak formulation of (\ref{approxlinear}).   This yields
\begin{equation}\label{gradlogest2}-\int_{\Omega_j}  (\mathcal{A}_k\nabla u_k) \cdot \nabla \Bigl(\frac{\psi^2}{u_j}\Bigl) dx = -\Bigl\langle\sigma_k, \psi^2\Big\rangle.
\end{equation}
On the other hand,
\begin{equation}\begin{split}\nonumber m_{j+1}\int_{\Omega_j} \frac{|\nabla u_k|^2}{u_k^2} \psi^2 dx \leq &-\int_{\Omega_j} (\mathcal{A}_k\nabla u_k) \cdot \nabla \Bigl(\frac{\psi^2}{u_k}\Bigl) dx \\
&+ 2M_{j+1} \int_{\Omega_j} \frac{|\nabla u_k|}{u_k} |\nabla\psi| \psi dx,
\end{split}\end{equation}
and therefore Cauchy's inequality yields a constant $C=C(M_{j+1},m_{j+1})>0$ such that
\begin{equation}\label{gradlogest3}\int_{\Omega_j} \frac{|\nabla u_k|^2}{u_k^2} \psi^2 dx \leq -C\int_{\Omega_j} (\mathcal{A}_k\nabla u_k) \!\cdot \!\nabla \Bigl(\frac{\psi^2}{u_k}\Bigl) dx + C\int_{\Omega_j} |\nabla \psi|^2 dx.
\end{equation}
Combining (\ref{gradlogest2}) and (\ref{gradlogest3}), we deduce (\ref{gradlogest}).
\end{proof}

From Lemma \ref{lemlog}, we deduce two estimates, depending on the addition properties of $\sigma$.  First, we deduce that the so-called logarithmic Cacciopolli inequality holds if $\sigma$ in addition satisfies (\ref{introlower}).

\begin{lem}\label{lemlogcacc}  Suppose that the real-valued distribution $\sigma$ satisfies (\ref{introlower}) with constant $\Lambda>0$, and let $(u_j)_j$ be as in (\ref{approxlinear}).  Let $\psi \in C^{\infty}_0(\Omega_j)$, then for any $k\geq j$
\begin{equation}\label{logcaccio}
\int_{\Omega_j} \frac{|\nabla u_k|^2}{u_k^2} \psi^2 dx \leq C(M_{j+1}, m_{j+1}, \Lambda)\int_{\Omega_j} |\nabla \psi|^2 dx
\end{equation}
\end{lem}
\begin{proof}  From Lemma \ref{lemlog}, it clearly suffices to show that, for all $\psi\in C^{\infty}_0(\Omega_j)$ we have
$$- \int_{\Omega_j} \psi^2 d\sigma_k \leq M_{j+1} \Lambda\int_{\Omega_j} |\nabla \psi|^2 dx.
$$
But this follows in precisely the same manner as Lemma \ref{mollem}.
\end{proof}
Recalling the definition of capacity in (\ref{capdef}), we arrive at the following corollary of Lemma \ref{lemlogcacc}:
\begin{cor}\label{capcor}  Suppose the hypotheses of Lemma \ref{lemlogcacc} are satisfied.  Then there exists a positive constant $C=C(\Lambda, m_{j+1}, M_{j+1})$, so that whenever $F\subset\subset \Omega_j$, we have
\begin{equation}\label{capcorst}
\int_{F} \frac{|\nabla u_k|^2}{u_k^2} dx \leq C{\rm{cap}}(F, \Omega_{j+1}), \quad \text{ for all }k\geq j.
\end{equation}
\end{cor}

In the case when $\sigma$ only satisfies (\ref{locsobball}), a local Morrey space estimate holds, which is a weakened version of (\ref{logcaccio}):

\begin{lem} \label{localmorrey} Suppose $n\geq 3$, and that $\sigma$ satisfies (\ref{locsobball}).  Consider the sequence $\{u_j\}$ as in (\ref{approxlinear}).  Then for each ball $B(x,r)$ so that $B(x,2r)\subset\Omega_j$, it follows that for all $k> j$
\begin{equation}\label{logballest}
\int_{B(x,r)} \frac{|\nabla u_k|^2}{u_k^2} \, dx \leq C(\Omega_j, M_{j+1}, m_{j+1}) r^{n-2}.
\end{equation}
\end{lem}

\begin{proof}  Fix such a ball $B(x,r)$ as in the statement of the lemma, and let $\psi \in C^{\infty}_0(B(x,2r))$ such that $\psi \equiv 1$ on $B(x,r)$, $0\leq \psi \leq 1$ and $|\nabla \psi|\leq C/R$. The lemma follows from estimating (\ref{gradlogest}) with this choice of $\psi$.  For $k> j$, it suffices to prove that
\begin{equation}\label{logballest1}
\Bigl| \int_{\Omega_k} \psi^2 \, d\sigma_k \Bigl| \leq C(\Omega_j) r^{n-2}.
\end{equation}
Picking $U = \Omega_j$ in the definition of (\ref{locsobball}),  Lemma \ref{molball} yields
$$\Bigl| \int_{\Omega_k} \psi^2 d\sigma_k \Bigl| \leq C_{(\Omega_j)_{\varepsilon_k}} r^{n-2}||\nabla (\psi^2)||_{L^2(\Omega_j)} \leq 2C_{(\Omega_j)_{\varepsilon_j}}||\nabla \psi||_{L^2(B(x,2r))}.
$$
The last inequality follows since $0\leq \psi\leq 1$.  Here $(\Omega_j)_{\varepsilon_j}$ is the $\varepsilon_j$-neighbourhood of $\Omega_j$.  Note that $(\Omega_j)_{\varepsilon_j} \subset \Omega_{j+1}$, and since $\varepsilon_k<\varepsilon_j$, it follows $(\Omega_j)_{\varepsilon_k} \subset (\Omega_j)_{\varepsilon_j}$ and hence by definition in (\ref{locsobball}) we have $C_{(\Omega_j)_{\varepsilon_k}}\leq C_{(\Omega_j)_{\varepsilon_j}}$.

The display (\ref{logballest1}) now follows from the estimate on the gradient of $\psi$.
\end{proof}

In the case $n=1$ or $2$; note that if $\sigma \in L^{-1,2}_{\text{loc}}(\Omega)$, then for all $k>j$ we have
\begin{equation}\label{lowngrad}
\int_{B(x,r)}\frac{|\nabla u_k|^2}{u_k^2} dx\leq C(\Omega_j, M_{j+1}, m_{j+1}) \text{ whenever }B(x,2r)\subset \Omega_j.
\end{equation}
In fact this estimate holds for all dimensions, but it is not strong enough to provide us with a uniform bound in higher dimensions.  The estimate (\ref{lowngrad}) follows from display (\ref{gradlogest}) in Lemma \ref{lemlog}.  Indeed, for $k>j$, just pick the test function $\psi$ in (\ref{gradlogest}) so that $\psi\equiv 1$ on $\Omega_j$, and $\psi\in C^{\infty}_0(\Omega_{j+1})$.  This yields
$$\int_{\Omega_j} \frac{|\nabla u_k|^2}{u_k^2} dx\leq C(\Omega_j, M_{j+1}, m_{j+1}, ||\sigma||_{L^{-1,2}(\Omega_{j+1})}).
$$
Here we are using the fact that the mollification does not effect the local dual Sobolev norm within $\Omega_j$, which can be established precisely as in Lemma \ref{molball}.  The estimate (\ref{lowngrad}) clearly follows from the previous display.

\subsection{A local gradient estimate}  The key estimate is the following:
\begin{prop}\label{mainprop}Suppose $\sigma$ is a real-valued distribution defined on $\Omega$, satisfying (\ref{introupper}) with $0<\lambda<1$, and in addition suppose that (\ref{locsobball}) holds.  Let $\{u_j\}$ be the sequence in (\ref{approxlinear}).
Then, whenever $\omega \subset\subset \Omega_j$, the following estimate holds:
\begin{equation}\label{apriori}
\int_{\omega} |\nabla u_k|^2 dx\leq C(\omega, \lambda, \Lambda, m_{j+1}, M_{j+1}, B, \Omega_j), \quad \text{ for all } \, k> j.
\end{equation}
\end{prop}

Note that the estimate (\ref{apriori}) is \textit{independent of $k$} for $k\geq j$.  This is the key to allow us to deduce the existence of positive solutions to (\ref{schrointro}).


\begin{proof}  Fix $j$, and $k$, as in the statement of the proposition, and let $v= u_k$.  It suffices to prove that whenever $B(x, 8r)\subset\subset \Omega_j$, there exists a positive constant $C>0$, depending on $n$, $\lambda$, $\Lambda$, $m_{j+1}$, $M_{j+1}$, $B$, $B(x,r)$ and $\Omega_j$ such that
\begin{equation}\label{suffintcond}
\int_{B(x,r)} |\nabla v|^2 dx\leq C.
\end{equation}
The reader should keep in mind that \textit{all constants will be independent of $k$}.
Fix such a ball $B(x, 8r)\subset\subset \Omega_j$.  To prove (\ref{suffintcond}), we will employ Proposition \ref{locdoub} in $U = \Omega_j$ to show that $v^2$ is doubling in $\Omega_j$, with constants independent of $k$.  To verify the hypothesis of Proposition \ref{locdoub}, we first show that $v^2$ satisfies a weak reverse H\"{o}lder inequality, i.e. the inequality (\ref{wrh}) holds in $\Omega_j$.  To this end, let us fix $B(z,2s)\subset\subset \Omega_j$.  Suppose first $n\geq 3$.  Let $\psi \in C^{\infty}_{0}(\Omega_j)$, then applying the Sobolev inequality yields
\begin{equation}\label{sobineq}
\Bigl(\int_{\Omega_j}v^{\frac{2n}{n-2}} |\psi|^{\frac{2n}{n-2}}dx\Bigl)^{\frac{n-2}{n}}\leq C \int_{\Omega_j} |\nabla v|^2\psi^2 \, dx + C\int_{\Omega_j} v^2 \, |\nabla\psi|^2 \, dx.
\end{equation}
Applying Lemma \ref{lemcacc} in the first term on the right hand side of (\ref{sobineq}) results in
\begin{equation}\label{caccsob}\Bigl(\int_{\Omega_j}v^{\frac{2n}{n-2}} |\psi|^{\frac{2n}{n-2}}dx\Bigl)^{\frac{n-2}{n}} \leq C \int_{\Omega_j}v^{2}|\nabla \psi|^2 dx.
\end{equation}
We now specialise (\ref{caccsob}) to the case $\psi\in C^{\infty}_0(B(z, 2s))$, with $\psi\equiv1 $ in $B(z,s)$, and $|\nabla\psi|\leq C/s$.  As a result,  we obtain
\begin{equation}\label{l2est}\Bigl(\dashint_{B(z,s)}(v^2)^{\frac{n}{n-2}}dx\Bigl)^{\frac{n-2}{n}} \leq  C\dashint_{B(z,2s)}v^{2} \, dx.
\end{equation}
The constant in $C>0$ in (\ref{l2est}) depends on $n, M_{j+1}, m_{j+1}$, and $\lambda$.  Hence, if $n\geq 3$, (\ref{wrh}) holds in $U = \Omega_j$, with $w=v^2$ and $q=n/(n-2)$.

If $n=2$, we slightly modify the above argument. The following Sobolev inequality is standard (see e.g. \cite{MZ97}, Corollary 1.57): for each $q<\infty$, and for all $f\in C^{\infty}_0(B(z,2s))$,
\begin{equation}\label{sobpeq2}
\Bigl(\dashint_{B(z,2s)} |f(y)|^q \, dy \Bigl)^{1/q} \leq C(q)\Bigl(\int_{B(z,2s)}|\nabla f(y)|^2 \, dy\Bigl)^{1/2}.
\end{equation}
Using (\ref{sobpeq2}) as in (\ref{sobineq}) and following the argument through display (\ref{l2est}), it follows in the case $n=2$ that (\ref{wrh}) holds in $U = \Omega_j$, with $w=v^2$ for any choice $q<\infty$.  Note that in the case $n=1$ even stronger Sobolev inequalities are at our disposal, and so the estimate (\ref{wrh}) continues to hold; we leave this to the reader.

To apply Proposition \ref{locdoub}, it remains to  show $\log(v)\in BMO(\Omega_j)$. For this let $B(z,2s)\subset \Omega_j$. The Poincar\'{e} inequality yields a constant $C=C(n)$ such that
\begin{equation}\begin{split}\nonumber
\dashint_{B(z,s)} |\log v & -\dashint_{B(z,s)} \log v|^2 \, dx\leq Cs^{2-n}\int_{B(z,s)} \frac{|\nabla u_k|^2}{u_k^2} \, dx.
\end{split}\end{equation}
First suppose $n\geq 3$.  Then from Lemma \ref{localmorrey},  we have
\begin{equation}\int_{B(z,s)} \frac{|\nabla u_k|^2}{u_k^2} \,  dx \leq C(M_{j+1}, m_{j+1}, \Omega_j)s^{n-2},
\end{equation}
 and hence,
\begin{equation}\label{bmoexplicit}
\dashint_{B(z,s)} |\log v  - \dashint_{B(z,s)} \log v|^2 dx \leq  C(M_{j+1}, m_{j+1}, \Omega_j).
\end{equation}
In the case $n=1,2$; we apply the weaker estimate (\ref{lowngrad}) in combination with Poincar\'{e}'s inequality, to conclude that (\ref{bmoexplicit}) remains true in these cases.  From (\ref{bmoexplicit}), we conclude (see (\ref{BMOdef})) that $\log v \in BMO(\Omega_j)$, with $BMO$-norm depending only on $n, m_{j+1}, M_{j+1}, \Omega_j$.  In particular, $v^2$ satisfies both (\ref{wrh}) and (\ref{bmo1}) in $\Omega_j$.  From Proposition \ref{locdoub}, it follows that $v^2$ is doubling in $\Omega_j$, with constants depending on $n, m_{j+1}, M_{j+1}, \Omega_j, \lambda$ and $\Lambda$, see (\ref{doub}).

Since $\Omega_j$ is a smooth connected set, one can find a Harnack chain from $B(x,2r)$ to the fixed ball $B\subset\subset \Omega_1$.  In other words, there are positive constants $c_0,\,c_1$ and $N>0$, depending on the smooth parameterization of $\Omega_j$, along with points $x_0, \dots x_N$ and balls $B(x_i, 4r_i)\subset \Omega_j$ so that
\begin{enumerate}
\item $B(x_0, r_0) = B(x,2r)$, and $B(x_N, r_N) = B$;
\item $r_i\geq c_0 \min(r_0, r_N)$, and $|B(x_i, r_i)\cap B(x_{i+1}, r_{i+1})| \geq c_1 \min(r_0, r_N)^n$ for all $i = 0\dots N-1$.
\end{enumerate}
Since $v^2$ is doubling in $\Omega_j$, a Harnack chain argument (see Remark \ref{chainarg}) applied to the chain construction above yields
\begin{equation}\nonumber\dashint_{B(x,2r)} v^2 dx \leq C(B(x, r), m_{j+1}, M_{j+1}, \Omega_j, B, \lambda, \Lambda) \dashint_B v^2 dx.
\end{equation}
It therefore follows from the normalization on $v^2$ that
\begin{equation}\label{vl2est}\dashint_{B(x,2r)} v^2 dx \leq C(B(x, r), m_{j+1}, M_{j+1}, \Omega_j, B, \lambda, \Lambda).
\end{equation}
To complete the proof, combine the Caccioppoli inequality (Lemma \ref{lemcacc}) with the estimate (\ref{vl2est}).  This results in the inequality
$$\int_{B(x, r)}|\nabla v|^2 dx \leq \frac{C}{r^2}\int_{B(x,2r)} v^2 dx\leq C,$$
for a constant $C>0$, depending on $n$, $m_{j+1}$, $M_{j+1}$, $B$, $\Lambda$, $\lambda$, $\Omega_j$ and $B(x, r)$.  Hence (\ref{suffintcond}) is proved for a constant independent of $k$.
\end{proof}

\subsection{Proof of Propositions \ref{refineexist} and \ref{linexist}} \label{existsubsec} We begin by proving Proposition \ref{refineexist}, before moving on to prove Proposition \ref{linexist}, and with it statement (i) of Theorem \ref{gensymthm}.
\begin{proof}[Proof of Proposition \ref{refineexist}]  Let $\Omega_j$ be an exhaustion of $\Omega$ by smooth domains.   We will use Proposition \ref{mainprop} repeatedly in each $\Omega_j$ to deduce the existence of a solution of (\ref{schrointro}).  Fix a ball $B\subset \Omega_1$, with $4B\subset \Omega_1$, and note that the construction of the approximate sequence from (\ref{approxlinear}) is valid under the present assumptions on $\sigma$, and the gradient estimate (\ref{apriori}) holds.

First, by (\ref{apriori}) with $j=1$, along with weak compactness and Rellich's theorem, we pass to a subsequence $\smash{{u}^{(1)}_j}$ of ${u}_j$ so that ${u}_j^{(1)} \rightarrow u^{(1)}$ weakly in $L^{1,2}_{\text{loc}}(\Omega_1)$,  and ${u}^{(1)}_j \rightarrow u^{(1)}$ both in $L^2(\Omega_1)$ and almost everywhere in $\Omega_1 $.  Let $\varepsilon_{j,1}$ be the corresponding sequence from (\ref{approxlinear}).
Since $\sigma \in L^{-1,2}(\Omega_1)$,  it follows that whenever $h\in C^{\infty}_0(\Omega_1)$
\begin{equation}\label{sigmaconv1}
\langle \sigma{u}^{(1)}_j, h \rangle = \langle \sigma,{u}^{(1)}_j h\rangle \rightarrow \langle \sigma,u^{(1)} h\rangle = \langle \sigma u^{(1)}, h\rangle. \end{equation}
Note also, by combining the uniform bound (\ref{apriori}), with convergence of the mollification, we have
\begin{equation}\begin{split}\nonumber|\langle \sigma_{\varepsilon_{j,1}}{u}^{(1)}_j, h \rangle-\langle \sigma {u}^{(1)}_j, h \rangle| & \leq ||\sigma_{\varepsilon_{j,1}}-\sigma||_{L^{-1,2}(\Omega_1)}||\nabla({u}^{(1)}_j h)||_{L^2}\\
& \leq C ||\sigma_{\varepsilon_{j,1}}-\sigma||_{L^{-1,2}(\Omega_1)} \rightarrow 0 \text{ as }j\rightarrow \infty.
\end{split}\end{equation}
We conclude:
\begin{equation}\label{sigmaconv}\langle \sigma_{j,1}{u}^{(1)}_j, h \rangle \rightarrow  \langle \sigma u^{(1)}, h\rangle.\end{equation}

Similarly, by linearity and local boundedness of the operator $\mathcal{A}$, we deduce that
$$\int_{\Omega_1} \mathcal{A}\nabla{u}_j^{(1)} \cdot \nabla h \, dx \rightarrow \int_{\Omega_1} \mathcal{A}\nabla u^{(1)} \cdot \nabla h \, dx.
$$
From the uniform bound (\ref{apriori}), along with standard convergence properties of the mollification in $L^2$:
$$\Bigl|\int_{\Omega_1}( \mathcal{A}_{j,1}-\mathcal{A})\nabla {u}_j^{(1)} \cdot \nabla h \, dx\Bigl|\rightarrow 0 \quad \text{ as }j\rightarrow \infty.
$$
It therefore follows that the limit function $u^{(1)}\in L^{1,2}_{\text{loc}}(\Omega_1)$ satisfies
\begin{equation}\label{1stsub}-{\rm{div}}(\mathcal{A} \nabla u^{(1)}) =\sigma u^{(1)} \quad \text{ in } \mathcal{D}'(\Omega_1).
\end{equation}
Given $\{{u}_k^{(j)}\}_{k}$, let us apply estimate (\ref{apriori}) in $\Omega_{j+1}$ to obtain a subsequence ${u}^{(j+1)}_k$ of $u^{(j)}_k$ and $u^{(j+1)}\in L^{1,2}_{\text{loc}}(\Omega_{j+1})$ with:
$${u}_k^{(j+1)} \rightarrow u^{(j+1)} \text{ weakly in } L^{1,2}_{\text{loc}}(\Omega_{j+1}),  \text{ and } {u}^{(j+1)}_k \rightarrow u^{(j+1)} \text{ a.e. in } \Omega_{j+1} .
$$
Note that by Lemma \ref{localsoblem}, it follows that $\sigma\in L^{-1,2}(\Omega_j)$.   As in the argument leading to (\ref{1stsub}), we see that $u^{(j+1)}$ satisfies
\begin{equation}\label{localsol2}-{\rm{div}}(\mathcal{A} \nabla u^{(j+1)}) =\sigma u^{(j+1)} \text{ in } \mathcal{D}'(\Omega_{j+1}).
\end{equation}
By construction, $u^{(j)} = u^{(j+1)}$ in $\Omega_j$.  Hence one can define $u\in L^{1,2}_{\text{loc}}(\Omega)$ by: $u = u^{(j)}$ in $\Omega_j$.  By (\ref{localsol2}) it follows that
$$-{\rm{div}}(\mathcal{A} \nabla  u) = \sigma u \quad \text{ in } \mathcal{D}'(\Omega).
$$
Next, let us demonstrate that $u$ is not the zero function.  To see this note that
\begin{equation}\label{2nonzero}\int_B({u}^{(\ell)}_j)^2 dx = 1, \quad \text{ for all }j,\,\ell.
\end{equation}
Since ${u}^{(\ell)}_j \rightarrow u$ in $L^2_{\text{loc}}(\Omega_{\ell})$, we may pass to the limit in (\ref{2nonzero}).  A standard application of Mazur's lemma shows that the limit solution $u\geq0$.  On the other hand, for any $k>0$, it follows from Lemma \ref{lemlog} and weak compactness that there exists $v\in L^{1,2}_{\text{loc}}(\Omega)$ so that $\log({u}_j^{(k)})\rightarrow v$ almost everywhere.  But then $v=\log(u)$ a.e. and therefore (see for example Theorem 1.32 of \cite{HKM}) $\log({u}_j^{(k)})$ converges weakly to $\log(u)$ in $L^{1,2}_{\text{loc}}(\Omega_k)$.  Hence $u>0$ quasi-everywhere, and $u$ is a positive weak solution of (\ref{schrointro}).\end{proof}

We now move onto Proposition \ref{linexist}.  Recall that here the matrix $\mathcal{A}$ satisfies the global ellipticity and boundedness conditions, so in all the previous estimates of this section, we replace $m_j=m$ and $M_j=M$.

\begin{proof}[Proof of Proposition \ref{linexist}]  Let us keep the notation from the proof of Proposition \ref{refineexist}.  The existence of a positive solution $u\in L^{1,2}_{\text{loc}}(\Omega)$ of (\ref{schrointro}) follows from Proposition \ref{refineexist}.  It was proved above in addition that $\log(u)$ is well defined in $L^{1,2}_{\text{loc}}(\Omega)$, and in each $\Omega_k$, $\log(u)$ is the weak limit of a sequence $\log u _j^k$.  From Lemma \ref{lemlogcacc}, it follows that, for all $\psi\in C^{\infty}_0(\Omega_k)$,
\begin{equation}\label{locallogest}
\int_{\Omega} \frac{|\nabla {u}_j^k|^2}{({u}_j^{(k)})^2} \psi^2 \, dx \leq C(M, m, \Lambda)\int_{\Omega} |\nabla \psi|^2 \, dx.
\end{equation}
Since $\nabla \log u_j^{(k)} $ converges to $\nabla \log u$ weakly in $L^2_{\text{loc}}(\Omega_k)$, we deduce from weak lower semi-continuity of the $L^2$ norm that for all $\psi\in C^{\infty}_0(\Omega_k)$
\begin{equation}\label{locallogest2}
\int_{\Omega} \frac{|\nabla u|^2}{u^2} \psi^2 \, dx \leq C(M, m, \Lambda)\int_{\Omega} |\nabla \psi|^2 \, dx.
\end{equation}
For any $k>0$, the estimate (\ref{locallogest2}) holds with a uniform constant for smooth functions supported in $\Omega_k$, and hence (\ref{logmult}) holds.
\end{proof}

\subsection{A logarithmic change of variable: solutions of (\ref{ricintro}) from solutions of (\ref{schrointro})}\label{logchange}
This section is concerned with deducing solutions of (\ref{ricintro}) from solutions of (\ref{schrointro}) by a logarithmic substitution.  This substitution is classical, for instance it appears in the study of ODEs in \cite{Hi48}, and there are examples that show it can be delicate, see e.g. \cite{FM00}.   In \cite{AHBV}, there is a rather comprehensive account of the connection between these two types of equations when $\sigma$ is a finite measure.  We will prove the following lemma, from which statement $\rm{(ii)}$ in Theorem \ref{gensymthm} and the remainder of Theorem \ref{refine} follows.
\begin{lem}\label{logsublin}  Let $\Omega$ be an open set, and let $\mathcal{A}:\Omega \rightarrow \mathbf{R}^{n\times n}$ satisfy the local conditions (\ref{locelliptic}).  Suppose that $\sigma\in L^{-1,2}_{\text{loc}}(\Omega)$, and that there exists a positive solution $u$ of (\ref{schrointro}). Then $v = \log(u)\in L^{1,2}_{\text{loc}}(\Omega)$ is a solution of (\ref{ricintro}).
\end{lem}
\begin{proof} Fix $U\subset\subset \Omega$.  The first step is to prove that
\begin{equation}\label{logdefined}\int_{U} \frac{|\nabla u|^2}{u^2}dx \leq C(U, \sigma, n, p).
\end{equation}
Let $\varepsilon>0$, and let $V$ be such that $U\subset\subset V\subset\subset \Omega$.   For $h\in C^{\infty}_0(V)$, test the weak formulation of (\ref{schrointro}) with $\psi = h (u+\varepsilon)^{-1}\in L^{1,2}_c(V)$.  This yields
\begin{equation}\label{Schro2Ric}\int_{\Omega} \frac{\mathcal{A}(\nabla u)}{u+\varepsilon} \cdot \nabla h \,dx = \int_{\Omega} \frac{\mathcal{A}(\nabla u)\cdot\nabla u}{(u+\varepsilon)^2} h dx + \langle \sigma \frac{u}{u+\varepsilon}, h\rangle.
\end{equation}
Let us now estimate the third term on the right.  By assumption $\sigma \in L^{-1,2}(V)$, and hence there exists $\vec \Gamma \in (L^2(V))^n$ so that $\sigma = \text{div}(\vec\Gamma)$ in $V$.  Therefore
\begin{equation}
 \langle \sigma \frac{u}{u+\varepsilon}, h\rangle = \int_{\Omega} \frac{\nabla u\cdot \vec\Gamma}{u+\varepsilon} \Bigl(\frac{\varepsilon}{u+\varepsilon}\Bigl) h dx + \int_{\Omega} \frac{u}{u+\varepsilon} \nabla h \cdot \vec\Gamma dx.
\end{equation}
Since $\varepsilon /(u+\varepsilon) \leq 1$, it follows from Cauchy's inequality that for any $\delta>0$, we have
\begin{equation}\label{logdefinedstep1}| \langle \sigma \frac{u}{u+\varepsilon}, h\rangle| \leq \delta\int_{\Omega} \frac{|\nabla u|^2}{(u+\varepsilon)^2} h dx + C_{\delta} \int_{\Omega} |\vec\Gamma|^2 h dx + \int_{\Omega} |\nabla h| |\vec\Gamma| dx.
\end{equation}
Now, let  $\varphi\in C^{\infty}_0(V)$, $\varphi\geq 0$, $\varphi\equiv 1$ on $U$, and put $h=\varphi^2$  in (\ref{Schro2Ric}).   Rearranging, using the ellipticity and boundedness assumptions (\ref{elliptic}), we obtain
\begin{equation}\begin{split}\label{logalmostdef}m_V \int_{\Omega}& \frac{|\nabla u|^2}{(u+\varepsilon)^2}\varphi^2 \, dx \leq   2M_V\int_{\Omega} \frac{|\nabla u|}{u+\varepsilon} \, |\nabla \varphi| \, \varphi \,  dx\\
&+  \delta\int_{\Omega} \frac{|\nabla u|^2}{(u+\varepsilon)^2} \varphi^2 \, dx + C_{\delta} \int_{\Omega} |\vec\Gamma|^2 \varphi^2 \, dx + 2\int_{\Omega} |\nabla \varphi| \,  |\vec\Gamma| \, \varphi \, dx.
\end{split}\end{equation}
Here the bound (\ref{logdefinedstep1}) has also been used.
Appealing to Cauchy's inequality again in (\ref{logalmostdef}), we obtain
$$\int_{\Omega} \frac{|\nabla u|^2}{(u+\varepsilon)^2} \varphi^2  \, dx \leq C(U, \sigma, n, p).
$$
Letting $\varepsilon \rightarrow 0$, (\ref{logdefined}) follows from Fatou's lemma.

Now, let us again look at (\ref{Schro2Ric}), this time with an arbitrary $h\in C^{\infty}_0(U)$.  It follows from (\ref{logdefined}) that as $\varepsilon \rightarrow 0$
$$\int_{\Omega} \frac{\mathcal{A}\nabla u}{u+\varepsilon} \cdot \nabla h \,dx \rightarrow \int_{\Omega} \frac{\mathcal{A}\nabla u}{u} \cdot \nabla h \,dx, \text{ and}$$
$$\int_{\Omega} \frac{(\mathcal{A}\nabla u) \cdot \nabla u}{(u+\varepsilon)^2} \, h \, dx\rightarrow \int_{\Omega} \frac{(\mathcal{A}\nabla u) \cdot \nabla u }{u^2}\,  h \, dx.
$$
To handle the last term in (\ref{Schro2Ric}),  note that from (\ref{logdefined}) and the dominated convergence theorem
$$\nabla \Bigl(\frac{u}{u+\varepsilon}\Bigl) = \Bigl(\frac{\varepsilon}{u+\varepsilon}\Bigl)\cdot\frac{\nabla u}{u+\varepsilon}\rightarrow 0  \text{ in }L^2(\Omega) \text{ as } \varepsilon \rightarrow 0,
$$
on the other hand, it is clear that $\frac{u}{u+\varepsilon} \rightarrow 1$ in $L^2(U)$,  as $\varepsilon \rightarrow 0.$
Thus, it follows that $\frac{u}{u+\varepsilon} \rightarrow 1$   in $L^{1,2}(U)$  as $\varepsilon \rightarrow 0.$
But since $\sigma \in L^{-1,2}(U)$ we conclude that
$$\langle \sigma \frac{u}{u+\varepsilon}, h\rangle = \langle \sigma, \frac{u}{u+\varepsilon} h\rangle \rightarrow\langle \sigma, h\rangle, \text{ as } \varepsilon \rightarrow 0.
$$
It follows that $v = \log(u)$ is a distributional solution of (\ref{ricintro}).  \end{proof}

\begin{proof}[Proof of Theorem \ref{gensymthm}, statement {\rm(ii)}]  This is nothing more than a restatement of Lemma \ref{logsublin} above, along with the trivial observation that if $u$ satisfies (\ref{class1}), then $v=\log(u)$ satisfies (\ref{class2}).\end{proof}

We may now also complete the proof of Theorem \ref{refine}.
\begin{proof}[Proof of Theorem \ref{refine}]  By Lemma \ref{locmorlocsobsame} and Proposition \ref{refineexist}, it follows that under the hypothesis of Theorem \ref{refine}, there exists a position solution $u\in L^{1,2}_{\text{loc}}(\Omega)$ of (\ref{schrointro}).  By Lemma \ref{logsublin}, setting $v=\log(u)$, we see that $v\in L^{1,2}_{\text{loc}}(\Omega)$ is a solution of (\ref{ricintro}).
\end{proof}

We did not use (\ref{introlower}) or global ellipticity and boundedness assumptions in the previous lemma, doing so allows us to conclude that the solution satisfies an additional multiplier condition.

\begin{lem}\label{logsubgrad} Under the assumptions of Lemma \ref{logsublin}, suppose that $\mathcal{A}$ satisfies (\ref{elliptic}), and $\sigma$ satisfies (\ref{introlower}) for a positive constant $\Lambda>0$.  Then there exists a solution $v\in L^{1,2}_{\text{loc}}(\Omega)$ of (\ref{ricintro}) satisfying (\ref{class2}).
\end{lem}

\begin{proof}
We will keep the notation from the proof in Lemma \ref{logsublin}.  It is left to prove that $v$ satisfies (\ref{class2}).  To this end, let us again test (\ref{schrointro}) with $\varphi^2/(u+\varepsilon) $, for a smooth test function $\varphi\in C^{\infty}_0(\Omega)$, $\varphi\geq 0$.  There exist constants $C_1, C_2>0$ depending on $m$ and $M$, so that
\begin{equation}\label{linembstep1}\int_{\Omega} \frac{|\nabla u|^2}{(u+\varepsilon)^2} \varphi^2 dx \leq  C_1 \int_{\Omega}|\nabla \varphi|^2 dx -  C_2\Bigl\langle \sigma \varphi\sqrt{\frac{u}{u+\varepsilon}},\varphi\sqrt{\frac{u}{u+\varepsilon}} \Bigl\rangle.
\end{equation}
Indeed, as in display (\ref{Schro2Ric}), one obtains by testing equation (\ref{schrointro}) the following identity:
\begin{equation}\label{Schro2Ric2} 2 \int_{\Omega} \varphi \frac{\mathcal{A}(\nabla u)}{u+\varepsilon} \cdot \nabla \varphi \,dx = \int_{\Omega} \frac{\mathcal{A}(\nabla u)\cdot\nabla u}{(u+\varepsilon)^2} \varphi^2 dx + \langle \sigma \frac{u}{u+\varepsilon}, \varphi^2\rangle
\end{equation}
Hence, from the ellipticity and boundedness assumptions (\ref{elliptic}), we have
\begin{equation}\label{apreselliptic}
m \int_{\Omega} \frac{|\nabla u|^2}{(u+\varepsilon)^2} \varphi^2 \, dx \leq M \int_{\Omega}\frac{|\nabla u|}{u+\varepsilon}\varphi  \cdot |\nabla \varphi| \, dx -  \langle \sigma \frac{u}{u+\varepsilon}, \varphi^2\rangle.
\end{equation}
From an elementary application of Cauchy's inequality in (\ref{apreselliptic}), display (\ref{linembstep1}) follows.   One can pick, for instance, $C_1 = (M/m)^2$ and $C_2 = 2/m$.

Next, applying (\ref{introlower}), it follows that the second term on the right hand side in (\ref{linembstep1}) is bounded by a constant multiple of
\begin{equation}\label{linembstep2}\Lambda \int_{\Omega}\frac{|\nabla u|^2}{u^2} \Bigl(\frac{\varepsilon}{u+\varepsilon}\Bigl)^2 \varphi^2 dx + \Lambda \int_{\Omega} \frac{u}{u+\varepsilon} |\nabla \varphi|^2 dx.
\end{equation}
The first term in (\ref{linembstep2}) converges to zero as $\varepsilon \rightarrow 0$, by virtue of (\ref{logdefined}) and the dominated convergence theorem.  Again by dominated convergence, the second term in (\ref{linembstep2}) converges to
$\int_{\Omega} |\nabla \varphi|^2 dx,$ as $\varepsilon\rightarrow 0.$
Substituting these estimates into display (\ref{linembstep1}), we deduce that (\ref{class2}) holds.
\end{proof}

\subsection{Existence of solutions to (\ref{ricintro}) implies the validity of (\ref{imbintro})}  \label{rictoformbd}

\begin{proof}[Proof of Theorem \ref{gensymthm}, statement {\rm(iii)}]  This will follow immediately from Lemmas \ref{symrefine} through \ref{backwithgrad} below.
\end{proof}


\begin{lem}\label{symrefine}  Let $\Omega$ be an open set, and suppose that $\sigma$ is a distribution defined on $\Omega$.  Let $\mathcal{A}$ be an $n\times n$ real-valued symmetric matrix defined on $\Omega$ satisfying (\ref{elliptic}).  Suppose there exists a supersolution $v\in L^{1,2}_{\text{loc}}(\Omega)$ of (\ref{ricintro}), then (\ref{introupper}) holds with $\lambda=1$.
\end{lem}

\begin{proof} Suppose that there exists a solution $v$ of (\ref{ricintro}), and let $\varphi\in C^{\infty}_0(\Omega)$, then testing (\ref{ricintro}) with $\varphi^2$, yields
$$\langle \sigma \varphi, \varphi \rangle  = \langle \sigma, \varphi^2 \rangle  \leq  2 \int_{\Omega} |\varphi||(\mathcal{A}(\nabla v))\cdot\nabla \varphi |dx - \int_{\Omega} (\mathcal{A}\nabla v)\cdot(\nabla v)\varphi^2 dx.
$$
Under the present assumptions, $\mathcal{A}$ is a symmetric positive definite matrix.  It follows that for $\xi, \eta \in \mathbf{R}^n$,
\begin{equation}\label{symineqvec}|(\mathcal{A}\xi)\cdot \eta| \leq ((\mathcal{A}\xi)\cdot \xi)^{1/2}((\mathcal{A}\eta)\cdot \eta)^{1/2}.
\end{equation}
Thus, we have that
\begin{equation}\begin{split}2\int_{\Omega} |\varphi||(\mathcal{A}(\nabla v))&\cdot\nabla \varphi| dx \leq 2\int_{\Omega} |\varphi| ((\mathcal{A}\nabla v)\cdot \nabla v)^{1/2}((\mathcal{A}\nabla \varphi)\cdot \nabla\varphi)^{1/2} dx \\
& \leq \int_{\Omega} \varphi^2 ((\mathcal{A}\nabla v)\cdot \nabla v) + \int_{\Omega}((\mathcal{A}\nabla \varphi)\cdot \nabla\varphi) dx.
\end{split}\end{equation}
Therefore (\ref{introupper}) holds with $\lambda = 1$.
\end{proof}

On the other hand, if symmetry is not assumed, one may still conclude the validity of (\ref{introupper}) and (\ref{introlower}) from (\ref{ricintro}), as we will show now.

\begin{lem} \label{nonsymcase} Let $\Omega$ be a connected open set, and suppose that $\mathcal{A}:\Omega\rightarrow \mathbf{R}^{n\times n}$, satisfying (\ref{elliptic}).  Suppose that $v\in L^{1,2}_{\text{loc}}(\Omega)$ is a supersolution of (\ref{ricintro}), then (\ref{introupper}) holds with
$$\lambda = \Bigl(\frac{M}{m}\Bigl)^2.$$
\end{lem}

\begin{proof}
Let us first show that (\ref{introupper}) holds with the given choice of $\lambda$.  To this end, let $\varphi \in C^{\infty}_0(\Omega)$ and test the weak formulation (\ref{schrointro}) with the valid test function $\varphi^2$.
Together with the assumptions (\ref{elliptic}), this yields
\begin{equation}\begin{split}\label{embtest}\langle \sigma \varphi, \varphi \rangle  = \langle \sigma, \varphi^2 \rangle& \leq M2 \int_{\Omega} |\nabla v| |\varphi||\nabla \varphi| dx - m\int_{\Omega} |\nabla v|^2\varphi^2 dx\\
& \leq \frac{M^2}{m}\int_{\Omega} |\nabla \varphi|^2 dx \leq \Bigl(\frac{M}{m}\Bigl)^2\int (\mathcal{A}\nabla \varphi) \cdot \nabla \varphi,
\end{split}\end{equation}
where Young's inequality was used in the last line.
\end{proof}




We conclude this section by showing that if the solution $v$ of (\ref{ricintro}) in addition satisfies (\ref{class2}), then $\sigma$ satisfies (\ref{introlower}) for a positive constant $\Lambda>0$.

\begin{lem}\label{backwithgrad}  Under the assumptions of either Lemma \ref{symrefine} or Lemma \ref{nonsymcase}, if one in addition assumes the solution $v\in L^{1,2}_{\text{loc}}(\Omega)$ satisfies (\ref{class2}),  then $\sigma$ satisfies (\ref{introlower}) for a positive constant $\Lambda>0$.
\end{lem}

\begin{proof}
Let $\varphi \in C^{\infty}_0(\Omega)$ and test the weak formulation (\ref{schrointro}) with the valid test function $\varphi^2$.  Then
\begin{equation}\begin{split}
\langle &\sigma \varphi, \varphi \rangle   = 2 \int_{\Omega} ((\mathcal{A}\nabla v)\cdot \nabla \varphi )\varphi dx - \int_{\Omega} |\nabla v|^2 \varphi^2 dx\\
& \geq -(M+1)\int_{\Omega} |\nabla v|^2 \varphi^2 - \int_{\Omega}|\nabla \varphi|^2 dx \geq -\frac{\Lambda}{m}\int_{\Omega} |\nabla \varphi|^2 dx,
\end{split}\end{equation}
for a suitable choice of constant $\Lambda>0$.  Appealing to (\ref{elliptic}), the result follows.
\end{proof}

\subsection{On the equation (\ref{ricintro})}\label{naturalclass} In this section, we make a few comments regarding our results for the equation (\ref{ricintro}) in comparison to existing literature, as the introduction of this paper focussed rather more on the Schr\"{o}dinger type equation (\ref{schrointro}).  First, we will restate a theorem which has been proved for reference:

\begin{thm}\label{necsufric} Let $\Omega$ be an open set.  Let $\sigma \in D'(\Omega)$ be a real-valued distribution.  Suppose $\mathcal{A}:\Omega \rightarrow\mathbf{R}^n$ is a symmetric real-valued matrix function satisfying (\ref{elliptic}).  The the following two statements hold:

\indent (i)  Suppose that $\sigma$ satisfies (\ref{introlower}) for a positive constant $\Lambda>0$, and (\ref{introupper}) for a constant $0<\lambda<1$, then there exists a positive solution $v\in L^{1,2}_{\text{loc}}(\Omega)$ of (\ref{ricintro}) so that (\ref{class2}) holds.  In addition, the constructed solution has the exponential integrability property: $e^v\in L^{1,2}_{\text{loc}}(\Omega)$.

\indent (ii)  Conversely, if there exists a solution $v\in L^{1,2}_{\text{loc}}(\Omega)$ of (\ref{ricintro}) so that (\ref{class2}) holds, then $\sigma$ satisfies (\ref{introlower}) for a positive constant $\Lambda>0$, and (\ref{introupper}) with $\lambda = 1$.
\end{thm}

In \cite{FM98}, solutions of (\ref{ricintro}) are proved in the global energy space $\Lo(\Omega)$ are proved when $\Omega$ is a bounded domain, under the assumption that $\sigma \in L^{n/2}(\Omega)$.  They explicitly note that this condition on $\sigma$ is used to guarantee that (\ref{imbintro}) is valid.  Theorem \ref{necsufric} therefore compliments their theorem with a more local result in nature, and therefore one which requires less restriction on $\sigma$.  As was noticed in \cite{FM98,FM00}, there exist classes of solutions of (\ref{ricintro}) that are exponentially integrable.  One can trace this principle back to the employment of certain nonlinear test functions in proving the existence of solutions to (\ref{ricintro}) (see e.g. \cite{Ev90, FM00}).  A refinement of this argument is what is also employed in the current paper, since we deduce Theorem \ref{necsufric} from our considerations of the Schr\"{o}dinger type equation via a logarithmic substitution.

The local exponential integrability in statement (i) is sharp, as can be seen from the example discussed in Section \ref{examples}.  The paper \cite{FM00} concerns quasilinear equations of $p$-Laplacian type, and we will consider such equations in our forthcoming paper \cite{JMV10a}.

\subsection{On the critical case $\lambda=1$.}\label{criticalcase}
In this subsection we briefly discuss the limiting case when $\lambda=1$. We shall prove the following proposition:

\begin{prop}\label{positivecrit}
Suppose that $\Omega$ is an open set, and suppose $\mathcal{A}$ is a symmetric matrix function satisfying (\ref{elliptic}).  Then (\ref{positive1}) holds if and only if there exists a positive superharmonic function $u$ such that
\begin{equation}\label{supersuper}
-{\rm div }(\mathcal{A}\nabla u) \geq \sigma u \quad  \text{ in }\Omega.
\end{equation}
\end{prop}

\begin{proof}  Let us assume that $\Omega$ is connected.  We shall assume $n\geq 2$, and leave the one dimensional case to the reader.  The necessity is well known, and holds even in very general potential theoretic frameworks, see e.g. \cite{Fit00}.  For the converse, let $\lambda_j \in (0,1)$ be a sequence such that $\lambda_j \rightarrow 1$.  Applying a very special case of Theorem \ref{gensymthm} above, we find a sequence of positive functions $\{u_j\}_j\in L^{1,2}_{\text{loc}}(\Omega)$ of
\begin{equation}\label{approxschr}
-\text{div}(\mathcal{A}\nabla u_j) =\lambda_j \sigma u_j,\; \text{ with }\dashint_B u_j \, dx = 1 .
\end{equation}
Here $B\subset\subset \Omega$ is a fixed ball.  We may assume that $u_j$ is superharmonic.
Next, let us fix a smooth connected subdomain $U$ of $\Omega$, so that $B\subset\subset U$.  For a fixed $q<n/(n-1)$, we will prove that for any ball $B(x,r)\subset\subset U$, the following estimate holds
\begin{equation}\label{locbound}\int_{B(x,r)} |\nabla u_j|^q \, dx \leq C(m, M, U, q, B(x,r)).
\end{equation}
To this end, note by the property (\ref{class1}) that
\begin{equation}\label{bmoemb}\int_{\Omega} \frac{|\nabla u_j|^2}{u_j^2} \, h^2 \, dx\leq C(m, M) \int_{\Omega}|\nabla h|^2 dx  \text{ for all }h\in C^{\infty}_0(\Omega).
\end{equation}
Display (\ref{bmoemb}) is also a standard property of superharmonic functions, as a consequence of Moser's work \cite{Mos60}.   From (\ref{bmoemb}) and the Poincar\'{e} inequality, one readily deduces as in Proposition \ref{mainprop} that $u_j\in BMO(U)$.  Therefore, as in Proposition \ref{locdoub}, from the John-Nirenberg inequality we find a constant $c=c(m, M, n)>0$, $0<c<1$, such that $u^c$ is doubling in $U$ (see (\ref{ld})), with doubling constants depending on $m, M$, and $n$.  For any $B(x,r)\subset\subset U$, a Harnack chain argument yields a constant $C=C((B(x,r), U, B, m, M)$ such that
\begin{equation}\label{smallubd}
\int_{B(x,r)} u^c dx \leq C \int_B u^c dx \leq C(B(x,r), U, B, m, M),
\end{equation}
where in the last equation, we have used normalization of $u_j$ in (\ref{approxschr}).  Let us now note that following inequality superharmonic functions, essentially due to Moser \cite{Mos60}: for $1<q<n/(n-1)$, there exists a constant $C=C(n,q)$ such that
\begin{equation}\label{wkharn}
\int_{B(x,r)}|\nabla u_j|^q \, dx \leq Cr^{n-q} \inf_{B(x,r)} u_j.
\end{equation}
Combining (\ref{wkharn}) and (\ref{smallubd}), the display (\ref{locbound}) follows.  Note that from (\ref{wkharn}) and Rellich's theorem, we deduce that there exists $u$ such that
$u_j \rightarrow u$ a.e in $\Omega$, and
$\dashint_{B} u\,dx =1.$
Furthermore, $u$ can be chosen to be superharmonic, by standard convergence properties, see \cite{KM}.

Our aim is to show that $u = \liminf_{j\rightarrow} u_j$ q.e.  To this end, let $v=\liminf_{j\rightarrow \infty} u_j$, and denote by $v^*$ the lower semi-continuous regularization of $v$.  By the fundamental convergence theorem for superharmonic functions, $v^*$ is superharmonic, and $v^*=\liminf_{j\rightarrow \infty} u_j$ quasi-everywhere (see Theorem 7.4 of \cite{HKM}).  Since $v^*=u$ a.e. and they are both superharmonic, we have that $u=v^*$ everywhere.  The claim follows.

Let us now conclude the argument.  Since $\sigma$ satisfies (\ref{positive1}), it does not charge sets of capacity zero.  By Fatou's lemma, for any $\varphi\in C^{\infty}_0(\Omega)$ with $\varphi\geq0$ it follows that
$$\liminf_{j\rightarrow\infty}\int_{\Omega} u_j \varphi \, d\sigma \geq \int_{\Omega} u\varphi \, d\sigma.
$$
Combining this last display with the weak convergence of the Riesz measure for superharmonic functions, see e.g. \cite{HKM}, we conclude that (\ref{supersuper}) holds.
\end{proof}

\section{Form boundedness}\label{formbdsec}

 In this section we apply Theorem \ref{gensymthm} to deduce a new proof of Theorem \ref{formbounded} below, which was the primary theorem in \cite{MV1} (see also \cite{MV5}).   Let us consider the case $\Omega = \mathbf{R}^n$, $n \ge 3$, since the case of a general domain, under
 certain mild restrictions on $\Omega$, can be reduced to the entire space, as was explained in \cite{MV1}.

\begin{thm}\label{formbounded} Let $\sigma \in D'(\mathbf{R}^n)$, $n \ge 3$. Then the  following statements hold.

{\rm (i)} The quadratic form inequality
\begin{equation}\label{quadineq}
| \langle \sigma, \, h^2 \rangle | \le C \, || \nabla h||^2_{L^2(\mathbf{R}^n)}, \quad \text{ for all }h \in C^\infty_0(\mathbf{R}^n),
\end{equation}
is valid if and only if $\sigma$ can be represented in the form
\begin{equation}\label{represent1}
\sigma = {\rm{div}} \, {\vec \Gamma},
\end{equation}
where $\vec \Gamma \in L^2_{{\rm loc}}(\mathbf{R}^n)^n$ obeys
\begin{equation}\label{properties1}
\int_{\mathbf{R}^n} h^2  |\vec \Gamma|^2  \, dx \le C_1 \, || \nabla h||^2_{L^2(\mathbf{R}^n)}, \quad \text{ for all } h \in C^\infty_0(\mathbf{R}^n).
\end{equation}

{\rm (ii)}  If $\sigma$ satisfies {\rm (\ref{quadineq})},  then  {\rm (\ref{represent1})} holds with $\vec \Gamma = \nabla ( \Delta^{-1} \sigma)$, where $\Delta^{-1} \sigma \in {\rm BMO} (\mathbf{R}^n)$, and
\begin{equation}\label{properties2}
\int_{\mathbf{R}^n} h^2  |\nabla (\Delta^{-1} \sigma)|^2  \, dx \le C_1 \, || \nabla h||^2_{L^2(\mathbf{R}^n)}, \quad  \text{ for all } h \in C^\infty_0(\mathbf{R}^n).
\end{equation}

{\rm (iii)}  If  {\rm (\ref{properties1})}  holds with $C_1=\frac 1 4$ then {\rm (\ref{quadineq})} holds with $C=1$. Conversely, if $\sigma$
satisfies {\rm (\ref{introupper})} with the upper
form bound $\lambda<1$, and  {\rm (\ref{introlower})} with the lower form bound $\Lambda >0$, then {\rm (\ref{properties2})}  holds with a
 constant $C_1$ which does not depend on
$\sigma$.

\end{thm}

\begin{proof}[Proof of Theorem \ref{formbounded}] The sufficiency part of  statement (i)  with
$C_1=C^2/4$ in inequality (\ref{properties1})  follows using integration by parts and Cauchy's inequality: if $\sigma = {\rm{div}} \, {\vec \Gamma}$,
then
\begin{equation}\begin{split}
|  \langle \sigma, \, h^2 \rangle | & = 2 \, \left \vert  \int_{\mathbf{R}^n} \vec \Gamma \cdot \nabla h \, h \, dx \right \vert \\
&\le 2 \,  || h  \vec \Gamma ||_{L^2(\mathbf{R}^n)} ||\nabla h||_{L^2(\mathbf{R}^n)}\le C \, ||\nabla h||^2_{L^2(\mathbf{R}^n)}.
\end{split}\end{equation}

To deduce the remainder of the theorem, we apply part (ii) of Theorem \ref{gensymthm} with $\tilde \sigma  := \sigma/(\lambda+\varepsilon)$, where $\lambda$ is the upper form bound of $\sigma$, and $\varepsilon>0$, so that  the
corresponding upper form bound $\tilde \lambda$ of $\tilde \sigma$ satisfies $\tilde \lambda <1$.  This  yields the existence of a
weak solution $\Psi \in L^{1,2}_{{\rm loc}}(\mathbf{R}^n)$ of the multi-dimensional Riccati equation:
\begin{equation}\label{3.7}
- \Delta \Psi  = | \nabla \Psi|^2 + \tilde \sigma  \quad {\rm in} \quad D'(\mathbf{R}^n).
\end{equation}
Using $h^2$, where $h \in C^\infty_0(\mathbf{R}^n)$, as a test function in this equation, and integrating by parts, we estimate
\begin{equation}\begin{split}\nonumber
\int_{\mathbf{R}^n} h^2 \, | \nabla \Psi|^2 dx &= 2 \int_{\mathbf{R}^n} \nabla \Psi \cdot \nabla h \, h \, dx - \langle  \tilde \sigma, h^2 \rangle \\
&\le 2  \,  || h  \nabla \Psi ||_{L^2(\mathbf{R}^n)} ||\nabla h||_{L^2(\mathbf{R}^n)} + \tilde \Lambda \,  ||\nabla h||_{L^2(\mathbf{R}^n)}^2,
\end{split}\end{equation}
where $\tilde \Lambda = \Lambda/(\lambda + \varepsilon)$ is the lower form bound of $\tilde \sigma$.
From this it follows:
\begin{equation}\label{psi-square}
\int_{\mathbf{R}^n} h^2 \, | \nabla \Psi|^2 dx \le (1+ \sqrt{\tilde \Lambda})^2 \, ||\nabla h||_{L^2(\mathbf{R}^n)}^2,  \text{ for all } h \in C^\infty_0(\mathbf{R}^n).
\end{equation}
In other words,
$|\nabla \Psi| \in M(L^{1,2}_0(\mathbf{R}^n) \to L^2(\mathbf{R}^n) )$,
where $M(L^{1,2}_0(\mathbf{R}^n) \to L^2(\mathbf{R}^n) )$ is the space of pointwise multipliers
from the homogeneous Sobolev space $L^{1,2}_0(\mathbf{R}^n)$ into $L^2(\mathbf{R}^n)$
defined in Sec.~\ref{notation}.
 Hence, $\tilde \sigma$ can be represented in the form
$$ \tilde \sigma = - {\rm div} \, \nabla \Psi - | \nabla \Psi|^2, \quad |\nabla \Psi| \in M(L^{1,2}_0(\mathbf{R}^n) \to L^2(\mathbf{R}^n) ).$$
Moreover, one can deduce from  {\rm (\ref{quadineq})}   that $\Delta^{-1} \, \tilde \sigma $ is well
defined in the sense of the weak-$ \star$ ${\rm BMO}$ convergence (see details in \cite{MV5}):
$$
\Delta^{-1} (\psi_N \, \tilde \sigma) \overset{weak-\star} {\longrightarrow} \Delta^{-1} \tilde \sigma \in {\rm BMO} (\mathbf{R}^n),
$$
where $\psi_N (x)= \psi( |x|/N)$, and $\psi\in C^\infty_0(\mathbf{R})$ is a standard cut-off function. It follows from  {\rm (\ref{psi-square})}
that $\Delta^{-1}  \, ( | \nabla \Psi|^2) \in {\rm BMO} (\mathbf{R}^n)$, and hence
\begin{equation}\label{3.8}
\Psi  =  - \Delta^{-1} \, ( | \nabla \Psi|^2)  -  \Delta^{-1} \tilde \sigma \in {\rm BMO} (\mathbf{R}^n).
\end{equation}
Thus, $\tilde \sigma$ can be represented in the form
\begin{equation}\label{3.9}
\tilde \sigma = {\rm div} \, \vec \Gamma, \qquad \vec \Gamma =   \nabla \Delta^{-1} \, \tilde \sigma,
\qquad {\rm in} \quad D'(\mathbf{R}^n),
\end{equation}
where
\begin{equation}\label{3.10}
 \Delta^{-1} \, \tilde \sigma =  - \Psi - \Delta^{-1} \, (  | \nabla \Psi|^2) \in {\rm BMO} (\mathbf{R}^n).
\end{equation}

To complete the proof of statements (ii) and (iii), it remains to verify
$$ \nabla \Delta^{-1} \, \tilde \sigma \in M(L^{1,2}_0(\mathbf{R}^n) \to L^2(\mathbf{R}^n) ).$$
Let
$
g = (-\Delta)^{-\frac 1 2} \, | \nabla \Psi|^2 \ge 0.$
 In other words, $g$ is the Riesz potential of order $1$ of $ | \nabla \Psi|^2$, so that
$| \nabla \Delta^{-1} \, (  | \nabla \Psi|^2) (x) | \le c(n) \, g(x)$  for almost every $x\in \mathbf{R}^n.$

Since $ |\nabla \Psi| \in M(L^{1,2}_0(\mathbf{R}^n) \to L^2(\mathbf{R}^n) )$, it follows (see Theorem 1.7 in \cite{Ver99}) that
$g \in M(L^{1,2}(\mathbf{R}^n) \to L^2(\mathbf{R}^n) )$.
Hence, $\nabla \Delta^{-1} \, (  | \nabla \Psi|^2) \in M(L^{1,2}_0(\mathbf{R}^n) \to L^2(\mathbf{R}^n) )$.
Thus,
$$\vec \Gamma =   \nabla \Delta^{-1} \, \tilde \sigma \in M(L^{1,2}_0(\mathbf{R}^n) \to L^2(\mathbf{R}^n) ).$$
This is equivalent to (\ref{properties2}),  and the proof is complete.   \end{proof}

\section{Semi-boundedness}\label{sembdsec}

Let $\Omega\subseteq \textbf{R}^{n}$ be an open set with $n\geq 1$, and suppose $\mathcal{A}$ is a matrix function satisfying (\ref{elliptic}).  In this section, we will consider real-valued distributions $\sigma \in \mathcal{D}'(\Omega)$ which are semi-bounded; that is, the quadratic form of the operator
$\mathcal{H} = -\text{div} (\mathcal{A} \nabla \cdot) - \sigma$ is non-negative:
\begin{equation}\label{semibd}
\langle \sigma, h^2 \rangle \leq \int_{\Omega} (\mathcal{A}\nabla h) \cdot \nabla h \, dx, \text{ for all }h\in  C^{\infty}_0(\Omega).
\end{equation}

It had been conjectured that a necessary and sufficient condition for (\ref{semibd}) to hold is the following condition:  there exist ${\vec{\Gamma}}\in L^2_{\text{loc}}(\Omega)^n$ and a constant $C>0$ so that $\sigma \leq \text{div}(\vec\Gamma)$, and
\begin{equation}\label{conj}
 \int_{\Omega} |h|^2 |\vec\Gamma|^2 \, dx\leq C \int_{\Omega}(\mathcal{A}\nabla h)\cdot \nabla h \, dx \text{ for all }h\in C^{\infty}_0(\Omega).
\end{equation}
A simple estimate using integration by parts and Cauchy's inequality in the form (\ref{symineqvec}) shows that condition (\ref{conj})
with $C=\frac 1 4$  is
sufficient for (\ref{semibd}) to hold. However, it is \textit{not necessary}, with any $C>0$, for (\ref{semibd}).  We defer the proof of this fact to Proposition \ref{conjfalse} below.  On the other hand, the following theorem provides a characterization of semi-bounded distributions.

\begin{thm}\label{semithm}  Let $\Omega$ be an open set, and let $\sigma\in \mathcal{D}'(\Omega)$ be a real valued distribution.  In addition, let $\mathcal{A}$ be a symmetric matrix function defined on $\Omega$ satisfying (\ref{elliptic}).  Then (\ref{semibd}) holds if and only if there exists a vector field ${\vec\Gamma}\in L^2_{\text{loc}}(\Omega)$ so that
\begin{equation}\label{semibdcond}
\sigma \leq {\rm{div}}(\mathcal{A}{\vec \Gamma}) - (\mathcal{A}\vec\Gamma) \cdot \vec\Gamma \qquad \text{in }\mathcal{D}'(\Omega).
\end{equation}
\end{thm}

There is an extension of Theorem \ref{semithm} to non-symmetric matrices $\mathcal{A}$.  Indeed the necessity of the condition (\ref{semibdcond}) extends to the non-symmetric case, see Proposition \ref{semiprop} below.  On the other hand, a repetition of the proof of Lemma \ref{nonsymcase} shows: if (\ref{semibdcond}) holds, then (\ref{semibd}) holds with a constant $(M/m)^2$ introduced in the right hand side.  Here $m$ and $M$ are the ellipticity constants from (\ref{elliptic}).

\begin{rem}  The proof of Theorem \ref{semithm} shows that (\ref{semibd}) holds if and only if there exist solutions to the differential inequality:
\begin{equation}\label{ricineq}-\text{div}(\mathcal{A}\nabla u) -(\mathcal{A}\nabla u)\cdot \nabla u \geq \sigma \qquad \text{ in }\Omega.
\end{equation}
The inequality in (\ref{ricineq}) cannot be strengthened to an equality for general distributions $\sigma\in \mathcal{D}'(\Omega)$.  Indeed, if there exists a solution $v\in L^{1,2}_{\text{loc}}(\Omega)$ of the equation
$$-\text{div}\mathcal{A}\nabla u = (\mathcal{A}\nabla u)\cdot \nabla u + \sigma \qquad\text{in }\Omega,
$$
then it follows that $\sigma \in L^{-1,2}_{\text{loc}}(\Omega)+L^{1}_{\text{loc}}(\Omega).$  For instance, when $n \ge 3$, one can pick $\sigma = -\delta_{x_0}$
 for $x_0\in \Omega$, then obviously  (\ref{semibd}) holds but $\sigma$ does not lie in the aforementioned class.  In fact, in the special case when $\sigma$ is a measure, it is known that $\sigma \in L^{-1,2}_{\text{loc}}(\Omega)+L^{1}_{\text{loc}}(\Omega)$ if and only if $\sigma$ does not charge sets of capacity zero (see Theorem 2.1 of \cite{BGO96}).
\end{rem}

Let us now move onto proving the Theorem:

\begin{proof}[Proof of Theorem \ref{semithm}] The sufficiency of (\ref{semibdcond}) for (\ref{semibd}) is a repetition of the proof of Lemma \ref{symrefine}.  The necessity of (\ref{semibdcond}) is somewhat more involved, and follows from Proposition \ref{semiprop} below.
\end{proof}

\begin{prop}\label{semiprop}Let $\Omega$ be an open set, with $\mathcal{A}$ a (possibly non-symmetric) matrix function defined on $\Omega$ satisfying (\ref{elliptic}).  Let $\sigma \in \mathcal{D}'(\Omega)$ satisfying (\ref{semibd}).  Then there exists $\vec \Gamma \in L^2_{\text{loc}}(\Omega)^n$ so that
\begin{equation}\label{semibdsuff}
\sigma \leq {\rm{div}}(\mathcal{A}{\vec \Gamma}) - (\mathcal{A}\vec\Gamma) \cdot \vec\Gamma \qquad \text{in }\mathcal{D}'(\Omega).
\end{equation}
\end{prop}

\begin{proof}[Proof of Proposition \ref{semiprop}] Without loss of generality, we may assume that $\Omega$ is connected.  Otherwise, we simply repeat the argument which follows in each component.\\
\textit{Step 1} (\textit{Approximation}).  Let $\Omega_j$ be an exhaustion of $\Omega$ by bounded smooth connected domains.  Let $\lambda_j\in (0,1)$ be any sequence so that $\lambda_j\rightarrow 1$, and
$$\varepsilon_j < 1/2\min(d(\Omega_j, \partial\Omega_{j+1}), 2^{-j}).$$
Consider $\sigma_j = \lambda_j \varphi_{\varepsilon_j}*\sigma$, and denote by $\mathcal{A}_j = \varphi_{\varepsilon_j}*\mathcal{A}$.  Then from Lemma \ref{mollem}, it follows that $\sigma_j$ satisfies
\begin{equation}\label{approxupper2}
\int_{\Omega_j}|h|^2 d\sigma_j \leq \lambda_j \int_{\Omega_j} \mathcal{A}_j(\nabla h)\cdot \nabla h \, dx, \quad \text{ for all }h\in C^{\infty}_0(\Omega_j).
\end{equation}
In addition to (\ref{approxupper2}), note the following estimate: for any $h\in C^{\infty}_0(\Omega_j)$
\begin{equation}\label{quallowbd2}\int_{\Omega_j}|h|^2 d\sigma_j\geq- C ||\sigma_j||_{L^{\infty}(\Omega_j)}|\Omega_j|^{2/n}||\nabla h||^2_2,
\end{equation}
this is a consequence of standard Sobolev inequalities (for all $n\geq 1$).  By (\ref{elliptic}), we conclude that
\begin{equation}\label{quallowbd}
\int_{\Omega_j}|h|^2 d\sigma_j \geq - C ||\sigma_j||_{L^{\infty}(\Omega_j)}|\Omega_j|^{2/n}\int_{\Omega_j} \mathcal{A}_j\nabla h \cdot \nabla h \, dx.
\end{equation}
From (\ref{approxupper2}) and (\ref{quallowbd}), we see that the hypothesis of Theorem \ref{gensymthm} are satisfied in $\Omega_j$ with potential $\sigma_j$.  It therefore follows from Theorem \ref{gensymthm} that there exists $v_j\in L^{1,2}_{\text{loc}}(\Omega_j)$ so that
\begin{equation}\label{vjapprox}
-\text{div}(\mathcal{A}_j\nabla v_j) = \mathcal{A}_j\nabla v_j\cdot \nabla v_j +\sigma_j, \qquad \text{in  } \mathcal{D}'(\Omega_j).
\end{equation}
By addition of a suitable constant, we may assume that, for a fixed ball $B\subset\subset \Omega_1$,
\begin{equation}\label{vknorm1}\Bigl| \int_B v_j dx\Bigl| =1, \quad \text{ for all  } j.
\end{equation}
\textit{Step 2} (\textit{A uniform bound}).  Fix $1\leq j \leq k$.  Our aim is to show that $v_k\in L^{1,2}_{\text{loc}}(\Omega_j)$, with constants independent of $k$. Let $h\in C^{\infty}_0(\Omega_j)$, by testing the weak formulation of $v_k$ in (\ref{vjapprox}) with $h^2$, we deduce from (\ref{elliptic}) that
\begin{equation}\nonumber m\int_{\Omega_k} |\nabla v_k|^2 h^2 dx \leq M \int_{\Omega_k}2h |\nabla v_k||\nabla h| - \int_{\Omega_k} h^2 d\sigma_k .
\end{equation}
Applying Cauchy's inequality in the first term on the right hand side:
$$m\int_{\Omega_k} |\nabla v_k|^2 h^2 dx \leq \frac{m}{2} \int_{\Omega_k} |\nabla v_k|^2 h^2 dx + 2\frac{M^2}{m}\int_{\Omega_k}|\nabla h|^2 dx - \int_{\Omega_k} h^2 d\sigma_k,
$$
and hence, as $\lambda_k \in (0,1)$, we have
\begin{equation}\label{vjbd1}
\int_{\Omega_k} |\nabla v_k|^2 h^2 dx \leq C\int_{\Omega_k}|\nabla h|^2 dx + C| \langle \varphi_{\varepsilon_k} * h^2, \sigma\rangle | .
\end{equation}
Next from standard distribution theory (see e.g. \cite{Str03}, Chapter 8), it follows that
$$| \langle \varphi_{\varepsilon_k} * h^2, \sigma\rangle | \leq C,
$$
for a constant $C$ depending on $\sigma$, the support of $h$ and $||\partial^{\alpha_{\ell}}h||_{L^{\infty}}$ for some collection of multi-indices $\alpha_1, \dots, \alpha_N$.  (One can see this from either the structure theorem, or by the definition of continuity).  In conclusion, for any $h\in C^{\infty}_0(\Omega_j)$
\begin{equation}\label{vjuniform}
\int_{\Omega_k} |\nabla v_k|^2 h^2 dx \leq C(\sigma, h).
\end{equation}
This proves the claim that $v_k\in L^{1,2}_{\text{loc}}(\Omega_j)$, with constants independent of $k$.

\textit{Step 3} (\textit{Conclusion}). This will be quite similar to Section \ref{existsubsec}.  Indeed, consider first $\Omega_1$.  Then from (\ref{vjuniform}) and weak compactness, we find a subsequence $v_{j,1}$ of $v_j$, and $v^1\in L^{1,2}_{\text{loc}}(\Omega_1)$ so that $v_{j,1}\rightarrow v^1$ weakly in $L^{1,2}_{\text{loc}}(\Omega_1)$.  From (\ref{vknorm1}) and an application of Rellich's theorem, the limit function $v_1$ is not identically infinite.

Let $\vec G\in (L^2(\Omega_1))^n\cap L^{\infty}$, with compact support in $\Omega_1$.  Note that
\begin{equation}\label{ajconv}\int_{\Omega}\mathcal{A}_{j,1}\nabla v_{j,1} \cdot \vec G \, dx = \int_{\Omega}((\mathcal{A}_{j,1}-\mathcal{A})\nabla v_{j,1}) \cdot \vec G \, dx  +\int_{\Omega}\mathcal{A}\nabla v_{j,1} \cdot \vec G\,  dx.
\end{equation}
The first term on the right of (\ref{ajconv}) converges to zero as $k\rightarrow \infty$.  Indeed, one can estimate:
$$ \Bigl|\int_{\Omega}((\mathcal{A}_{j,1}-\mathcal{A})\nabla v_{j,1}) \cdot \vec G \,  dx\Bigl|
$$
$$
\leq ||\vec G||_{\infty}\Bigl(\int_{\text{supp}(\vec G)}|\nabla v_{j,1}|^2dx\Bigl)^{1/2}\Bigl(\int_{\text{supp}(\vec G)}|\mathcal{A}_{j,1}-\mathcal{A}|^2 dx\Bigl)^{1/2},
$$
and the right hand side converges to zero by (\ref{vjuniform}) and standard properties of approximate identities.  For the second term on the right hand side of (\ref{ajconv}), note that by weak convergence
$$\int_{\Omega}\mathcal{A}\nabla v_{j,1} \cdot \vec G \,  dx \rightarrow \int_{\Omega} \mathcal{A}\nabla v^1\cdot \vec G \,  dx, \quad \text{ as } \, j\rightarrow \infty.$$
It follows that, for all $\vec G\in (L^2(\Omega_1))^n\cap L^{\infty}$, with compact support in $\Omega_1$,
\begin{equation}\label{ajconv2}\int_{\Omega}\mathcal{A}_{j,1}\nabla v_{j,1} \cdot \vec G \,
 dx \rightarrow \int_{\Omega}\mathcal{A}\nabla v^1 \cdot \vec G \,  dx, \quad \text{ as } \, j\rightarrow \infty.
\end{equation}
It is not difficult to see that one can extend (\ref{ajconv2}) for all $\vec G \in (L^2(\Omega_1))^n$, with compact support in $\Omega_1$.  In other words, that $\mathcal{A}_{j,1}\nabla v_{j,1} \rightarrow \mathcal{A}\nabla v^1$ weakly in $L^2_{\text{loc}}(\Omega_1)^n$.

Let $h\in C^{\infty}_0(\Omega_1)$ so that $h\geq 0$.  We next claim that
\begin{equation}\label{lowweaksemi}\liminf_{j\rightarrow\infty}\int_{\Omega}(\mathcal{A}_{j,1}\nabla v_{j,1})\cdot \nabla v_{j,1} h \, dx \geq \int_{\Omega} (\mathcal{A}\nabla v^1)\cdot \nabla v^1 h \, dx.
\end{equation}
To see this, denote by $\mathcal{A}^s$ the symmetric part of $\mathcal{A}$, i.e. $2\mathcal{A}^s = \mathcal{A} + \mathcal{A}^t$.  Then as before it follows $(\mathcal{A}_{j,1})^s\nabla v_{j,1}\rightarrow \mathcal{A}^s\nabla v^1$ weakly in $L^2_{\text{loc}}(\Omega_1)$.  In addition, by standard properties of mollification, $(\mathcal{A}_{j,1})^s\rightarrow \mathcal{A}^s$ in the weak-$\star$ topology of $L^{\infty}(\Omega_1)$.
We will repeatedly use the observation that the non-symmetric part does not contribute toward the quadratic form.  First note that as a result of the weak convergence, we have
\begin{equation}\begin{split}\int_{\Omega_1} (\mathcal{A}\nabla v^1)\cdot \nabla v^1 hdx &=\int_{\Omega_1} (\mathcal{A}^s\nabla v^1)\cdot \nabla v^1 hdx \\
&= \liminf_{j\rightarrow \infty}\int_{\Omega_1} (\mathcal{A}_{j,1})^s\nabla v_{j,1}\cdot \nabla v^1 hdx.
\end{split}\end{equation}
Next, by symmetry of the matrix (see (\ref{symineqvec})), we estimate for each $j$
\begin{equation}\begin{split}\nonumber\int_{\Omega_1} (\mathcal{A}_{j,1})^s\nabla v_{j,1}\cdot \nabla v^1 hdx \leq &\Bigl(\int_{\Omega_1} (\mathcal{A}_{j,1})^s\nabla v_{j,1}\cdot \nabla v_{j,1} h dx\Bigl)^{1/2} \\
&\cdot \Bigl(\int_{\Omega_1} (\mathcal{A}_{j,1})^s\nabla v^1\cdot \nabla v^1 h dx\Bigl)^{1/2}\end{split}\end{equation}
By taking the limit infimum of both sides, using the weak-$\star$ convergence of the convolution, one obtains (\ref{lowweaksemi}).  Here we are using the following elementary fact for two bounded sequences $(a_j)$ and $(b_j)$:
$$\text{If }\liminf_{j\rightarrow\infty}a_j = a\geq0 \text{ and }\lim_{j\rightarrow\infty}b_j = b\geq0, \text{ then }\liminf_{j\rightarrow\infty}a_j b_j \leq ab.$$
On the other hand, by standard properties of the convolution, and since $\lambda_j\rightarrow 1$:
$$\lambda_j\langle \sigma, \varphi_{\varepsilon_j}*h\rangle \rightarrow  \langle \sigma, h\rangle, \text{ as }j\rightarrow \infty.
$$
Keeping (\ref{vjapprox}) and (\ref{ajconv2}) in mind, we conclude that
\begin{equation}
-\text{div}(\mathcal{A}\nabla v^1) \geq (\mathcal{A}\nabla v^1)\cdot \nabla v^1 + \sigma \qquad \text{in }\mathcal{D}'(\Omega_1).
\end{equation}
For $k\geq 1$, and given the sequence $\{v_{j,k-1}\}$, a repetition of the above argument yields a subsequence $v_{j,k}$ of $v_{j,k-1}$ so that $v_{j,k}$ converges to $v^k\in L^{1,2}_{\text{loc}}(\Omega_k)$ with
\begin{equation}
-\text{div}(\mathcal{A}\nabla v^k) \geq (\mathcal{A}\nabla v^k)\cdot \nabla v^k  + \sigma \qquad \text{in }\mathcal{D}'(\Omega_k).
\end{equation}
Furthermore, as in Section \ref{existsubsec}, we may assert that $v^k = v^{k-1}$ in $\Omega_{k-1}$.  One can therefore define a function $v\in L^{1,2}_{\text{loc}}(\Omega)$ so that
$$-\text{div}(\mathcal{A}\nabla v) - (\mathcal{A}\nabla v)\cdot \nabla v \geq  \sigma \qquad \text{in }\mathcal{D}'(\Omega).
$$
To complete the proof, it suffices to let $\vec\Gamma = -\nabla v$.
\end{proof}

\section{The local regularity of solutions to the Schr\"{o}dinger equation with prescribed boundary values}\label{FNVsec}

The goal of this section is to apply the regularity techniques developed in this paper to the recent work of Frazier, Nazarov and Verbitsky \cite{FNV}.  The point here is to prove regularity of a given solution with prescribed boundary values, when we already know there exists a majorant of the given solution.

Let $\Omega\subset \textbf{R}^{n}$ be a bounded domain so that the boundary Harnack inequality is valid (for instance a Lipschitz, or more generally a NTA domain).
Let us now fix $x_0\in \Omega$, and let $G(x,x_0)$ be the Green's function for the Laplace operator relative to $\Omega$.  Then we define $m(x) = \min(1, G(x,x_0))$.  If $\Omega$ is a $C^{1,1}$ domain then $m(x)$ is pointwise comparable to $\text{dist}(x, \partial\Omega)$.

Let $\sigma$ be a locally finite Borel measure in $\Omega$.   Then the following theorem is proved:
\begin{thm}\label{FNV} \cite{FNV} Suppose that $\sigma$ satisfies the following embedding inequality
\begin{equation}\label{FNVemb}\int_{\Omega} h^2 d\sigma \leq \lambda \,  \int_{\Omega} |\nabla h|^2 dx, \text{ for all } h \in C^{\infty}_0(\Omega),
\end{equation}
with $0<\lambda<1$.  In addition, suppose that there is a constant $c>0$ so that
\begin{equation}\label{cond2}
\int_{\Omega} m(x)\exp\Bigl(\frac{c}{m(x)}\int_{\Omega} m(y) d\sigma(y)\Bigl) d\sigma(x)  <\infty.
\end{equation}
Then there is a solution $u_1$ of the equation
\begin{equation}\label{feykac}\begin{cases}
-\Delta u_1 = \sigma u_1 \text{ in } \Omega,\\
u_1 = 1 \text{ on }\partial \Omega.
\end{cases}\end{equation}
Conversely, if there is a solution of (\ref{feykac}), then (\ref{cond2}) holds for a positive constant $c=c(\Omega)$, and (\ref{FNVemb}) holds with $\lambda = 1$.
\end{thm}

The solution constructed in Theorem~\ref{FNV} is  interpreted in the potential theoretic sense,
i.e., $u_1\in L^1(\Omega, m d \sigma)$, and
\begin{equation*}u_1(x)= \int_\Omega G(x,y)  \, u_1 (y) \, d \sigma(y) +1.
\end{equation*}
  If $\Omega$ is a bounded $C^{1,1}$ domain then $u_1\in L^1(\Omega, dx)\cap L^1(\Omega, m d \sigma)$, and is  a solution to (\ref{feykac})  in the very weak sense (see \cite{FNV}).

 Our primary result in this section is the following:
\begin{thm}\label{u1reg}  Consider the solution $u_1$ of (\ref{feykac}) constructed in Theorem \ref{FNV}, then
\begin{equation}
u_1\in L^{1,2}_{\text{loc}}(\Omega).
\end{equation}
\end{thm}

\begin{proof}  Let $\Omega_j$ be a exhaustion of $\Omega$ by smooth domains.  Let $\varphi_j \in C^{\infty}_0(\Omega)$ be such that $\varphi_j \equiv 1$ on $\Omega_j$ and $0\leq \varphi_j \leq 1$.  Note that (as in Lemma \ref{localsoblem}) $\varphi_j \sigma \in L^{-1,2}(\Omega)$.  Since (\ref{FNVemb}) holds with constant $0<\lambda <1$, we follow a similar argument to Section \ref{approxlinesub}, using the Lax-Milgram lemma to obtain a unique $v_j\in L^{1,2}(\Omega)$ satisfying:
\begin{equation}\begin{cases}\label{FNVapprox}
-\Delta v_j = \varphi_j\sigma v_j \in \Omega,\\
\, v_j-1 \in \Lo(\Omega).
\end{cases}\end{equation}
Note here that $v_{j+1}\geq v_j$.

Let $B(x, 2r)\subset\subset\Omega$.  By repeating the proof of Lemma \ref{lemcacc}, we deduce that  there is a constant $C=C(n)$ so that
\begin{equation}\label{FNVgradest}
\int_{B(x,r)} |\nabla v_j|^2 \, dx \leq \frac{C}{r^2}\int_{B(x, 2r)} v_j^2 \,  dx.
\end{equation}
From (\ref{FNVgradest}), one asserts, as in displays (\ref{sobineq}) through (\ref{l2est}), that
\begin{equation}\label{FNVrh}\Bigl(\dashint_{B(x,r)}v_j^{2n/(n-2)} \, dx\Bigl)^{(n-2)/2} \leq C\dashint_{B(x,2r)} v_j^2 \, dx.
\end{equation}
Since (\ref{FNVrh}) holds for all balls $B(x,2r)\subset\subset \Omega$, we see that the hypothesis of Lemma \ref{reverseholder} are valid.  Applying the lemma with $t<n/(n-2)$, one finds a constant $C(t)>0$ so that
\begin{equation}\label{FNVrh1}\Bigl(\dashint_{B(x,r)}v_j^{2}dx\Bigl)^{1/t} \leq C(t)\Bigl(\dashint_{B(x,2r)} v_j^tdx\Bigl)^{1/t}.
\end{equation}
Combining (\ref{FNVrh1}) with (\ref{FNVgradest}), we conclude
\begin{equation}\label{FNVgradest1}
\int_{B(x,r)} |\nabla v_j|^2 dx \leq C(t)r^{n-2}\Bigl(\dashint_{B(x, 2r)} v_j^t \, dx\Bigl)^{1/t}.
\end{equation}
We now wish to show that $v_j\leq u_1$.  For a locally finite measure $\sigma$, denote by $\mathcal{G}^{\sigma}(x,y)$ the minimal Green's function of $-\Delta - \sigma$ (see \cite{FV1}), i.e. the minimal positive solution $u(\cdot,y)$ of the equation
$$-\Delta u(\cdot, y) - \sigma \, u ( \cdot ,y)=\delta_y, \quad y \in \Omega.
$$
Then $u_1 = 1 + \int_{\Omega}\mathcal{G}^{\sigma}(x,y)d\sigma(y)$, and $v_j =  1+ \int_{\Omega}\mathcal{G}^{\varphi_j \sigma}(x,y)\varphi_j(y) d\sigma(y)$ (recall $\varphi_j\sigma \in W^{-1,p'}(\Omega)$, so this representation coincides with the unique solution).  By construction $\mathcal{G}^{\sigma}$ is monotone in $\sigma$ (it can be represented by a Neumann series), and since $\varphi_j\sigma\leq \sigma$, it therefore follows that $v_j\leq u_1$, for each $j$.  Here $u_1$ is as defined in (\ref{feykac}).  Thus, from (\ref{FNVgradest1})
\begin{equation}\label{FNVgradest2}
\int_{B(x,r)} |\nabla v_j|^2 dx \leq Cr^{n-2}\Bigl( \frac{1}{|B(x, 2r)|}\int_{B(x, 2r)} u_1^t dx\Bigl)^{2/t}.
\end{equation}
Letting $t< n/(n-2)$, and recalling the weak Harnack inequality (valid since $\sigma \geq0$), we deduce the estimate:
\begin{equation}\label{FNVgradest3}
\int_{B(x,r)} |\nabla v_j|^2 \, dx \leq Cr^{n-2}\Bigl(\inf_{B(x,r)}u_1\Bigl)^2.
\end{equation}
Using (\ref{FNVgradest3}), we readily deduce that there exists $v\in L^{1,2}_{\text{loc}}(\Omega)$, so that $v_j$ increases to $v$, and $v_j \rightarrow v$ weakly in $L^{1,2}_{\text{loc}}$.  Furthermore, as in the proof of Proposition \ref{linexist}:
\begin{equation}\label{FNVdistsoln}-\Delta v = \sigma \, v \quad \text{ in } \mathcal{D}'(\Omega).
\end{equation}
Since $v_j \leq u_1$, it follows $v\leq u_1$.  On the other hand, we see that $v=1$ on $\partial\Omega$ in the potential theoretic sense.  Indeed, with $G(\mu)$ denoting the Green's function of the Laplacian relative to $\Omega$ applied to $\mu$, we see by (\ref{FNVdistsoln}) that
\begin{equation}\label{FNVabstract1}
v = G(\sigma v) + h,
\end{equation}
where $h$ is the greatest harmonic minorant of $v$.  Clearly, $h\geq 1$.  On the other hand, since $v\leq u_1$,
$$h-1 \leq v-1 \leq u_1-1 = G(\sigma u_1).
$$
Thus $h-1$ is a nonnegative harmonic minorant of $G(\sigma u_1)$.   However, by the Riesz decomposition theorem, the greatest harmonic minorant of $G(\sigma u_1)$ is zero.  Thus $h\equiv 1$, and $v=1$ on $\partial\Omega$ in the potential theoretic sense.

By minimality of $u_1$, it thus follows that $v=u_1$, and $u_1\in L^{1,2}_{\text{loc}}(\Omega)$.
\end{proof}

\section{Examples}\label{examples}

Our first result in this section completes our discussion of the condition (\ref{conj}), and that it does not provide a characterization of distributions satisfying (\ref{semibd}).

\begin{prop}\label{conjfalse}  Let $\Omega = \mathbf{R}^n$, and $\mathcal{A}$ be the $n\times n$ identity matrix.  Let $\sigma$ be the radial potential defined by:
$$ \sigma (x) = \cos r + \frac{n-1}{r}\sin r -\sin ^2 r, \quad r=|x|, \quad x \in  \mathbf{R}^n.
$$
Then $\sigma$ satisfies (\ref{semibd}), but cannot be represented in the form (\ref{conj}).
\end{prop}

\begin{proof}We first consider the case $n=1$ and $\Omega=\mathbf{R}_+=(0, +\infty)$. Note that a criterion of form boundedness takes the form \cite{MV1}, \cite{MV2}:  $\sigma= \Gamma'$ where
\begin{equation}\label{mv1}
\int_{a}^\infty | \Gamma|^2 dx \le \frac{C}{a}, \qquad a>0.
\end{equation}
Let $\sigma =\cos x - \sin^2 x$. Then $\sigma$ is semibounded by (\ref{semibdcond}), i.e.,
$$
\int_{\mathbf{R}_+}  h^2 \, \sigma \,  dx  \le \int_{\mathbf{R}_+} | h' |^2\,  dx, \quad h \in C^\infty_0(\mathbf{R}_+),
$$
but $\sigma$ cannot be represented in the form
\begin{equation}\label{form}
\sigma = \Gamma' - \mu, \quad \mu \ge 0,
\end{equation}
where $\Gamma$ satisfies (\ref{mv1}) with any $C>0$.

In fact, even a weaker condition
\begin{equation}\label{weak}
\int_a^{a+1} |\Gamma (x)|\, dx = o(1) \quad as \quad a \to +\infty
\end{equation}
cannot be satisfied if $\sigma$ is of the form (\ref{form}).

Indeed, suppose
\begin{equation}\label{hyp}
\cos x - \sin^2 x = \Gamma' -  \mu, \quad \mu \ge 0,
\end{equation}
Then
$$
\Gamma (x) = \sin x \, \left(1 + \frac 1 2 \cos x\right) - \frac {x}{2} + \varphi(x),
$$
where $\varphi (x)$ is nondecreasing on $\mathbf{R}_+$.

Let $\alpha_0= \arccos \frac{\sqrt 5-1}{2} \approx .904<1$ so that $\cos \alpha - \sin^2 \alpha \ge  0$
for $-\alpha_0\leq \alpha \leq \alpha_0$, and consequently  $\Gamma$ is nondecreasing in the interval $[2 \pi n - \alpha_0, \, 2 \pi n + \alpha_0]$. Hence, for $a = \frac {\alpha_0}{2}$, it follows
that
\begin{equation}\begin{split}\nonumber
\Gamma (a + 2 \pi n) &- \Gamma (2 \pi n) \\
&=  \sin a\, \left(1 + \frac 1 2 \cos a \right) - \frac {a}{2} + \varphi(a + 2 \pi n) - \varphi(2\pi n )\\
& \ge  \sin a\, \left(1 + \frac 1 2 \cos a \right) - \frac {a}{2} =C,
\end{split}\end{equation}
where $C$ is independent of $n$.  Here $C>0$ since  $\sin x\, \left(1 + \frac 1 2 \cos x \right) - \frac {x}{2}$ is increasing on $(-\alpha_0, \alpha_0)$ and  equals zero at the origin.

On the other hand, for  $\alpha =  a + 2 \pi n$ we have
$$
\Gamma (\alpha) \le \frac 2 {\alpha_0} \int_\alpha^{\alpha + a} \Gamma (x) \, dx  \le \frac 2 {\alpha_0} \int_\alpha^{\alpha +1} |\Gamma (x)| \, dx.
$$
Similarly, $
\Gamma (2 \pi n) \ge  \frac 2 \alpha_0 \int_{-a + 2 \pi n}^{2 \pi n} \Gamma (x) \, dx. $
Hence,
$$
\Gamma (a + 2 \pi n) -\Gamma (2 \pi n) \le  \frac 2 {\alpha_0} \int_{a + 2 \pi n}^{a + 2 \pi n +1} |\Gamma (x)| \, dx
+ \frac 2 {\alpha_0} \int_{-a + 2 \pi n}^{-a +2 \pi n +1} |\Gamma (x)| \, dx.
$$
By  (\ref{weak}) the right-hand side of the preceding inequality tends to zero as  $n \to +\infty$.
This  contradicts the estimate
$$
\Gamma (a + 2 \pi n) - \Gamma (2 \pi n) \ge C>0
$$
obtained above.

We now are in a position to consider the multi-dimensional case $\Omega=  \mathbf{R}^n$, $n \geq 3$. Let
\begin{equation}\label{ex}
\sigma = \cos r + \frac{n-1}{r} \sin r - \sin^2 r, \quad r =|x|, \quad x\in \mathbf{R}^n.
\end{equation}
Then by (\ref{semibdcond}) with $\Gamma = \frac{x}{r} \sin r$ it follows that
(\ref{semibd}) holds.

Note that $\frac{\sin r}{r}$ satisfies the inequality
\begin{equation}\label{by-parts}
\left \vert \int_{\mathbf{R}^n} h^2 \, \frac{\sin r}{r} \, dx \right \vert
\le C \, ||\nabla h||^2_{L^2(\mathbf{R}^n)}, \quad h \in C^\infty_0(\mathbf{R}^n).
\end{equation}
Indeed, using polar coordinates and integration by parts, we obtain
\begin{equation}\begin{split}
\Bigl| \int_{S^{n-1}}\int_{\mathbf{R}_+} & h(r \xi)^2 \,  \frac{\sin r}{r} \, r^{n-1} dr d \xi \Bigl|\\
&\leq  \Bigl| \int_{S^{n-1}}\int_{\mathbf{R}_+} 2h(r \xi) \, \nabla h(r \xi) \cdot \xi  \, \frac{\cos r}{r} \, r^{n-1} \, dr d \xi \Bigl|\\
&+(n-2) \left \vert\int_{S^{n-1}}\int_{\mathbf{R}_+} \cos r \, h^2(r \xi)  \, r^{n-3} \, dr d \xi \right \vert \\
\end{split}\end{equation}
Applying Cauchy's inequality and Hardy's inequality, we estimate the last line by
\begin{equation}\begin{split}\nonumber
 2 \Bigl(\int_{\mathbf{R}^n} \frac{h(x)^2}{r^2} \, dx \Bigl)^{\frac 1 2} &||\nabla h||_{L^2(\mathbf{R}^n)} + (n-2) \int_{\mathbf{R}^n} \frac{h(x)^2}{r^2}  dx\\
&\le C \, ||\nabla h||^2_{L^2(\mathbf{R}^n)},
\end{split}\end{equation}
and hence  (\ref{by-parts}) holds.

It remains to show that  $\sigma_1 =  \cos r - \sin^2 r$
does not satisfy the inequality $
\sigma_1 \le \rm{div} \, \Gamma,$
for any $\Gamma$ which obeys (\ref{conj}).  This is equivalent (see
\cite{MV1}),  to finding $\Gamma$ with
$\rm{div} \, \Gamma$ is form bounded. We note that
as was proved in \cite{MV1}, we may always pick
$\Gamma = \nabla \Phi$ where $\rm{div} \, \Gamma = \Delta \Phi$,
 and
\begin{equation}\label{phi}
\int_{\mathbf{R}^n}h^2 \, |\nabla \Phi |^2 \,  dx
\le C \, || \nabla h ||^2_{L^2(\mathbf{R}^n)}, \text{ for all }
 h \in C^\infty_0(\mathbf{R}^n),
\end{equation}
for some $C$ independent of $h$.   Suppose now that
\begin{equation}\label{rad}
\sigma_1 \le \Delta \Phi,
\end{equation}
where $\Phi$ satisfies (\ref{phi}). Since $\sigma_1$ is radially symmetric, it follows by using the average of $\Phi$ over the
unit sphere that (\ref{rad}) holds with a radially symmetric
 $\Phi_0(x) = \varphi(r)$ so that
$$
\sigma_1 (r) \le \varphi''(r) + \frac{n-1}{r} \varphi'(r), \quad r>0.
$$
Moreover,
$$\varphi'(r)=\frac{1}{|S^{n-1}|} \int_{S^{n-1}} \sum_{k=1}^n
\frac{\partial \Phi}{\partial x_k} (r \xi_k) \, \xi_k \, d \xi,
$$
and hence
$$
|\varphi'(r)|^2 \le \frac{1}{|S^{n-1}|} \int_{S^{n-1}} |\nabla \Phi(r \xi)|^2 \, d \xi.
$$
From this and (\ref{phi}) with  a radially symmetric test function $h$, we deduce
\begin{equation}\begin{split}\nonumber
  \int_{\mathbf{R}_+} h^2(r) \, |\varphi'(r)|^2 \, r^{n-1} dr &\le \frac{1}{|S^{n-1}|}  \int_{S^{n-1}} \int_{\mathbf{R}_+}  |\nabla \Phi(r \xi)|^2 \, h^2(r) \, r^{n-1} dr d \xi\\
 &\le C \,
 \int_{\mathbf{R}_+} |h'(r)|^2 \, r^{n-1} \, dr.
 \end{split}\end{equation}
Consequently,
$$
\int_0^a  |\varphi'(r)|^2 \, r^{n-1} \, dr \le C \, a^{n-2}, \quad a>0.
$$
We now let $\psi (r) = \varphi' (r)$. It follows from the preceding estimate that
$$
\int_a^{a+1} | \psi (r)| \, dr \le C \, a^{-\frac 1 2},
\quad a>0.
$$
This implies
$$
\int_a^{+\infty} \frac{| \psi(r)|}{r} dr \le C \, a^{-\frac 1 2},
\quad a>0.
$$
Next, denote by $g(r)$ the function:
$$g(r) = \psi(r) - (n-1) \int_r^{+\infty} \frac{\psi(r)}{r} dr.
$$
It is easy to see that $g$ satisfies the same inequality as $\psi$:
$$
\int_a^{a+1} | g (r)| \, dr \le C \, a^{-\frac 1 2},
\quad a>0.
$$
Furthermore,
$$
g'(r) =  \psi' (r) + \frac{n-1}{r} \, \psi(r) = \varphi'' (r) + \frac{n-1}{r} \, \varphi(r)= \Delta \Phi (x).
$$
Thus, $
\sigma_1 (r) = \cos r - \sin^2 r \le g'(r), \quad r>0,$
where $g$ satisfies the condition
$$
\int_a^{a+1} | g (r)| \, dr = o(1),
\quad a\to +\infty.
$$
Hence, by the one-dimensional example on $\mathbf{R}_+$ considered above with $g=\Gamma$, we arrive at a contradiction.
\end{proof}

In the following example we consider non-symmetric operators, with the aim to show that the non-symmetric part can effect the constant appearing in the form bound.  In Theorem \ref{gensymthm}, it was shown that when one has a solution to the Schr\"{o}dinger equation (\ref{schrointro}), or Riccati equation (\ref{ricintro}), whose operator $\mathcal{A}$ is non-symmetric, then $\sigma$ satisfies (\ref{introupper}) with $\lambda$ depending on the ellipticity constants (\ref{elliptic}).  This example shows that such a conclusion is not artificial:
\begin{exa}\label{nonsym}  Let $n=3$, and suppose $\mathcal{A} = I+B$, where $B$ has zero entries except for $b_{1,2} = Ca(x_1)$ for a constant $C$, and $a(x_1)$ a Lipschitz continuous function.  Suppose in addition $b_{2,1}=-b_{1,2}$.
If $u(x) = 1+ |x|^2$, then $u$ solves
$$-\text{div}(\mathcal{A}\nabla u) = \sigma u(x), \text{   with    }\sigma = \frac{- 6 + 2x_2 C a'(x_1)}{1+ |x|^2}
$$
On the other hand
$$\mathcal{A}(\xi)\cdot \xi = |\xi|^2.
$$
It follows that in the case of non-symmetric matrices $\mathcal{A}$, the constant in statement (iii) of Theorem \ref{gensymthm} depends on the constant $C$, and hence the operator $\mathcal{A}$. \end{exa}

The next example (which is  well known) demonstrates the sharpness of our primary theorem for Schr\"{o}dinger type equations.  In particular we confirm the assertions made in the introduction.

\begin{exa}Consider positive solutions $u$ of
\begin{equation}\begin{cases}\label{quadratic}
-\Delta u = \frac{c}{|x|^2} u \text{ in }\mathbf{R}^n,\\
\,\inf_{\mathbf{R}^n} u = 0,
\end{cases}\end{equation}
with $c\leq (n-2)^2/4$.  It is well known that (\ref{quadratic}) has positive solutions (up to constant multiple) of the form $u_{\pm}(x) =  |x|^{\alpha_{\pm}}$, where
\begin{equation}\label{defalpha}\alpha_{\pm} = \frac{2-n}{2} \pm \frac{1}{2} \sqrt{ (n-2)^2 - 4c}.
\end{equation}
If $c<(n-2)^2/4$, then by Hardy's inequality it follows that (\ref{introupper}) holds with $0<\lambda<1$.  For $0<c<(n-2)^2/4$ we see that by choosing $\alpha_+$, there exists a solution $u_+ \in L^{1,2}_{\text{loc}}(\Omega)$ of (\ref{quadratic}).   Taking $c$ arbitrarily close to $(n-2)^2/4$, we see that the existence of a solution $u_+\in L^{1,2}_{\text{loc}}(\Omega)$ of (\ref{quadratic}) is the optimal local regularity.   The same example shows that solutions need not be locally bounded, and therefore positive solutions of (\ref{quadratic}) do not satisfy the Harnack inequality.

Choosing $\alpha_{-}$ in (\ref{defalpha}), it follows that there exist positive solutions $u_-\in L^{1,1}_{\text{loc}}(\mathbf{R}^n)$ of (\ref{quadratic}), which are not in $L^{1,2}_{\text{loc}}(\mathbf{R}^n)$.

Finally, let  $c=(n-2)^2/4$ in (\ref{quadratic}). The resulting unique positive solution does not lie in $L^{1,2}_{\text{loc}}(\mathbf{R}^n)$.  This latter point shows the assumption $0<\lambda<1$ in (\ref{introupper}) is necessary in order to prove statement ${\rm(i)}$ of Theorem \ref{gensymthm}.  We remark that the uniqueness of the positive superharmonic solution in this case is known even for quasilinear generalizations of (\ref{quadratic}), see \cite{PS05}.
\end{exa}

\end{document}